\documentclass[11pt]{amsart}

\usepackage[T1]{fontenc}
\usepackage{lmodern}
\usepackage{microtype}
\usepackage{amsmath,amssymb,mathtools,mathrsfs}
\usepackage{enumitem}
\usepackage[colorlinks=true,citecolor=blue,linkcolor=blue,urlcolor=blue]{hyperref}
\usepackage[nameinlink,capitalise,noabbrev]{cleveref}

\newtheorem{theorem}{Theorem}[section]
\newtheorem{proposition}[theorem]{Proposition}
\newtheorem{lemma}[theorem]{Lemma}
\newtheorem{corollary}[theorem]{Corollary}

\newtheorem*{theoremA}{Theorem A}
\newtheorem*{theoremB}{Theorem B}
\newtheorem*{theoremC}{Theorem C}
\newtheorem*{theoremD}{Theorem D}
\theoremstyle{definition}
\newtheorem{definition}[theorem]{Definition}

\theoremstyle{remark}
\newtheorem{remark}[theorem]{Remark}

\DeclareMathOperator{\argmax}{argmax}
\DeclareMathOperator{\supp}{supp}
\DeclareMathOperator{\proj}{proj}
\DeclareMathOperator{\intt}{int}
\newcommand{\Prob}{\mathcal M}
\newcommand{\Inv}{\mathcal M}
\newcommand{\Ker}{\mathcal K}
\newcommand{\Eq}{\mathsf{Eq}}
\newcommand{\R}{\mathbb R}
\newcommand{\N}{\mathbb N}
\newcommand{\cE}{\mathcal E}
\newcommand{\cF}{\mathcal F}
\newcommand{\cO}{\mathcal O}
\newcommand{\cC}{\mathcal C}
\newcommand{\GammaT}{\Gamma_T}
\newcommand{\OmegaT}{\Omega_T}
\newcommand{\Ptop}{\Pi_{\mathrm{top}}}
\newcommand{\Pmeas}{\Pi}
\newcommand{\eps}{\varepsilon}
\newcommand{\weakstar}{\mathrm{w}^*}

\title[Nonlinear thermodynamic formalism for correspondences]
  {Nonlinear Thermodynamic Formalism for Correspondences}

\author{Dingxuan Tang}
\address{College of Science, Northwest A\&F University, Yangling, Shaanxi, China}
\email[Dingxuan Tang]{usagi.tang@gmail.com}
\author{Rui Yang}
\address{School of Mathematics, Northwest University, Xi'an, Shaanxi, China}
\email[Rui Yang]{zkyangrui2015@163.com}
\author{Zhiming Li}
\address[Corresponding author: Zhiming Li]{School of Mathematics, Northwest University, Xi'an, Shaanxi, China}
\email[Zhiming Li]{china-lizhiming@163.com}

\begin{document}
\raggedbottom

\begin{abstract}
We develop a nonlinear thermodynamic formalism for correspondences on compact
metric spaces. A continuous energy is evaluated on empirical measures of
successive pairs of states. The measure-theoretic pressure of a stationary
transition pair is the sum of its kernel entropy and the energy of the joint
distribution of two successive states. For forward
expansive correspondences satisfying an ergodic-abundance condition, we prove a
variational principle and the existence of nonlinear equilibrium pairs. We further prove that limits
of the two-coordinate projections of asymptotically maximizing Gibbs ensembles
are averages of the two-coordinate distributions of equilibrium pairs.
When the energy depends on finitely many potentials on the space of admissible
pairs, Legendre regularity assumptions allow us to identify nonlinear equilibrium
pairs among linear equilibrium pairs through a self-consistency equation and
a global maximization condition. Under these assumptions, real analyticity
in one dimension and uniqueness of linear equilibrium pairs imply finiteness
of the nonlinear equilibrium set. Finite-state examples
exhibit continuous and first-order transitions, metastability, and phases that
have identical one-coordinate marginals but distinct two-coordinate
distributions. A criterion based on Taylor expansions determines the
associated critical exponents.
\end{abstract}

\maketitle
\tableofcontents
\section{Introduction}

A correspondence $T$ on a compact metric space $X$ assigns  each point
$x\in X$  to a nonempty closed subset $T(x)\subset X$.  It  describes a
dynamics in which the next state need not be uniquely determined by the current
one.  The classical entropy--pressure--equilibrium framework under
expansiveness and specification properties goes back to Ruelle~\cite{ruelle-statistical-mechanics}.
Topological entropy for set-valued functions was developed in
\cite{kelly-tennant-set-valued-entropy}, and the orbit-space viewpoint used
there is closely related to the one adopted below.  Li, Li, and
Zhang~\cite{li-li-zhang-correspondences} developed a linear
thermodynamic formalism for such systems.  Their equilibrium object is a pair
$(\mu,Q)$ consisting of a transition probability kernel $Q$ supported by $T$
and a $Q$-stationary probability measure $\mu$.  They proved, in particular, a
variational principle for forward expansive correspondences.

For single-valued systems, the quadratic pressure and generalized Curie--Weiss
case were studied by Leplaideur and Watbled~\cite{leplaideur-watbled-quadratic}.
Buzzi, Kloeckner, and Leplaideur~\cite{buzzi-kloeckner-leplaideur-nonlinear}
developed a general formalism in which a continuous energy on probability
measures replaces the integral of a potential. Related weighted,
higher-dimensional, and Carath\'eodory-type variants appear in
\cite{barreira-holanda-higherdim,feng-huang-weighted-pressure,
yang-chen-zhou-weighted,ding-wang-nonlinear-pressure}.
We develop the corresponding theory for multivalued dynamics.

For a compact metric space $Y$, write $C(Y)=C(Y,\R)$ for  the space of  continuous
functions on $Y$,  equipped with the supremum $\|\cdot\|_\infty$. Let $\Prob(Y)$ denote the Borel probability
measures on $Y$, endowed with the weak-star topology in the sense that $\lambda_n\to\lambda$
if $$\int_Y f\,d\lambda_n\to\int_Y f\,d\lambda.$$ for every $f\in C(Y)$.
All convergence of probability measures below is in this topology unless
otherwise stated. For a Borel map $f:Y\to Z$ to a compact metric space $Z$, the pushforward of
$\lambda\in\Prob(Y)$ is defined by
$(f_*\lambda)(B):=\lambda(f^{-1}(B))$ for Borel $B\subseteq Z$.
If $S:Y\to Y$ is
continuous, write
\[
 \Inv(Y,S):=\{\nu\in\Prob(Y):S_*\nu=\nu\}
\]
for the set of $S$-invariant measures on $Y$.

We  first present our main results by clarifying several concepts. Let 
$$\GammaT:=\{(x,y)\in X^2:y\in T(x)\}$$
be the space of two-coordinate orbits defined 
in \cite[Section~2]{li-li-zhang-correspondences}. A  continuous potential for a
correspondence  depends on an admissible transition $(x,y)$ with
$y\in T(x)$, rather than on $x$ alone. Following \cite{buzzi-kloeckner-leplaideur-nonlinear, li-li-zhang-correspondences},  our energy is a continuous
map
\[
  \cE:\Prob(\GammaT)\longrightarrow\R,
\]
For an admissible orbit segment
$\boldsymbol{x}=(x_0,\ldots,x_n)$, the empirical measure of successive pairs is given by
\[
  L_n(\boldsymbol{x})
  :=\frac1n\sum_{j=0}^{n-1}\delta_{(x_j,x_{j+1})},
\]
where $\delta_(x,y)$ is the Dirac measure concentrated at the point $(x,y)\in \Gamma_T.$
Roughly speaking, the upper and lower nonlinear topologial pressures are  defined   by computing  the sum of these nonlinear weights
$\exp(n\cE(L_n(\boldsymbol{x})))$ over the  separated orbit segments.
If $Q$ is supported by $T$ and $\mu Q=\mu$, set
\[
  \gamma_{\mu,Q}(dx,dy):=\mu(dx)Q_x(dy).
\]
This is the joint distribution of the initial state and the next state of
the stationary Markov process with transition kernel $Q$ and stationary
measure $\mu$ \cite[Definition~5.8]{li-li-zhang-correspondences}. 
We write $\Ker(X;T)$ for the Borel transition kernels supported by $T$ and
$\Inv(X,Q):=\{\mu\in\Prob(X):\mu Q=\mu\}$, where
$(\mu Q)(A):=\int_X Q_x(A)\,d\mu(x)$ for every Borel set $A\subseteq X$.
Our nonlinear measure-theoretic pressure at $(\mu,Q)$ is defined by
\[
  \Pmeas_{\cE}(T;\mu,Q)
  :=h_\mu(Q)+\cE(\gamma_{\mu,Q}).
\]
The first main result  establishes a variational principle by showing the the nonlinear
topological pressure of an erergy  equals  the supremum of nonlinear measure-theoretic pressure over stationary transition pairs. 

\begin{theoremA}[Nonlinear variational principle]
Let $T$ be a forward expansive correspondence on a compact metric space and let
$\cE\in C(\Prob(\GammaT),\R)$.  If $(T,\cE)$ has an abundance of ergodic path
measures, then
\begin{equation*}
  \underline\Pi_{\mathrm{top}}^{\cE}(T)
  =\overline\Pi_{\mathrm{top}}^{\cE}(T)
  =\sup_{\substack{Q\in\Ker(X;T)\\ \mu\in\Inv(X,Q)}}
    \bigl\{h_\mu(Q)+\cE(\gamma_{\mu,Q})\bigr\}.
\end{equation*}
Their common value is dentoed by $\Ptop^{\cE}(T)$.
\end{theoremA}

Li--Li--Zhang's entropy is attached to a stationary transition pair,
whereas the nonlinear formalism of Buzzi--Kloeckner--Leplaideur is formulated
for invariant measures of a single-valued system.  Passing to the path space, the Markov replacement result in
\cref{sec:markovization}  returns the variational problem to stationary
transition pairs without losing entropy. Therefore, affine energies recover the linear
correspondence formalism of \cite{li-li-zhang-correspondences}, and
single-valued correspondences recover the framework of
\cite{buzzi-kloeckner-leplaideur-nonlinear}.

The assumption of the nonlinear ergodic-abundance  introduced in
\cref{def:abundance} cannot simply be omitted for a general continuous
energy; see~\cite{buzzi-kloeckner-leplaideur-nonlinear}. Two useful sources of the abundance hypothesis are  convex energies and entropy-density of ergodic
path measures given in 
\cref{prop:abundance-criteria}, where the latter applies to forward expansive correspondences
with specification property, see \cref{cor:specification-abundance}.

\begin{theoremB}[Equilibrium pairs and limits of Gibbs ensembles]
Let $T$ be a forward expansive correspondence on a compact metric space, let
$\cE\in C(\Prob(\GammaT),\R)$, and assume that $(T,\cE)$ has an abundance of
ergodic path measures.  Then nonlinear equilibrium pairs exist.  There is a
fixed expansive scale $\eps_*>0$ such that the following also holds.  If
$n_k\to\infty$ and $\cC_k\subset\cO_{n_k+1}(T)$ are
$(n_k,\eps_*)$-separated sets whose nonlinear partition sums have exponential
growth rate $\Ptop^{\cE}(T)$, then every accumulation point of the weighted
empirical measures $\Gamma_{\cC_k}$ defined in \cref{sec:gibbs} is an average
of the two-coordinate distributions of equilibrium pairs.  If all equilibrium
pairs have the same two-coordinate distribution $\gamma_*$, then
$\Gamma_{\cC_k}\to\gamma_*$ and their first marginals converge to the first
marginal of $\gamma_*$, in the weak-star topology.
\end{theoremB}


For energies of the form
\[
 \cE(\gamma)=F\left(\int_{\GammaT}\boldsymbol\psi\,d\gamma\right),
 \qquad
 \boldsymbol\psi=(\psi_1,\ldots,\psi_d)\in C(\GammaT,\R^d),
\]
with $F$ continuous on a neighborhood of the convex hull of
$\boldsymbol\psi(\GammaT)$, the infinite-dimensional variational problem reduces to a finite-dimensional
one.  We write
$\Ptop^{F,\boldsymbol\psi}(T):=\Ptop^{\cE}(T)$ for this energy and define
\begin{align*}
 \mathcal R(\boldsymbol\psi)
 &:=\left\{\int_{\GammaT}\boldsymbol\psi\,d\gamma_{\mu,Q}:
       Q\in\Ker(X;T),\ \mu\in\Inv(X,Q)\right\},\\
 \mathfrak h(\boldsymbol z)
 &:=\sup\left\{h_\mu(Q):
       \int_{\GammaT}\boldsymbol\psi\,d\gamma_{\mu,Q}
       =\boldsymbol z\right\}.
\end{align*}
In the definition of $\mathfrak h$, the supremum is over stationary
transition pairs with the indicated rotation vector.
Under the hypotheses of Theorem~C, every nonlinear equilibrium pair
is a linear equilibrium pair for a potential
$\boldsymbol y\cdot\boldsymbol\psi$, where
\begin{equation}\label{eq:intro-self-consistency}
  \boldsymbol y=\nabla F(\nabla P(\boldsymbol y)),
  \qquad
  P(\boldsymbol y):=P(T,\boldsymbol y\cdot\boldsymbol\psi).
\end{equation}
A solution yields a nonlinear equilibrium pair only when its rotation vector
$\nabla P(\boldsymbol y)$ globally maximizes $\mathfrak h+F$.
Strict local maxima that are not global maxima describe metastable states in
the finite-state examples below.

\begin{theoremC}[Finite-dimensional reduction]
Let $T$ be a forward expansive correspondence on a compact metric space, let
$\boldsymbol\psi\in C(\GammaT,\R^d)$, and suppose that
$(T,\boldsymbol\psi)$ satisfies the $C^1$ Legendre assumptions in
\cref{def6.2}.  Let $U$ be an open
neighborhood of the convex hull of $\boldsymbol\psi(\GammaT)$, and let
$F\in C^1(U)$.  Assume
abundance for the energy
$\cE(\gamma)=F(\int\boldsymbol\psi\,d\gamma)$ and suppose that the maximizing
set below lies in the interior of the rotation set.  Then
\begin{equation}\label{eq:intro-finite-reduction}
 \Ptop^{F,\boldsymbol\psi}(T)
 =\max_{\boldsymbol z\in\mathcal R(\boldsymbol\psi)}
   \{\mathfrak h(\boldsymbol z)+F(\boldsymbol z)\}.
\end{equation}
If $V$ denotes the maximizing set in
\eqref{eq:intro-finite-reduction} and
$Y=(\nabla P)^{-1}(V)$, the nonlinear equilibrium pairs are exactly
\[
 \bigcup_{\boldsymbol y\in Y}
 \Eq(T,\boldsymbol y\cdot\boldsymbol\psi).
\]
Here $\Eq(T,\varphi)$ denotes the set of linear equilibrium pairs for the two-coordinate potential $\varphi$.
Every $\boldsymbol y\in Y$ satisfies the self-consistency equation
\eqref{eq:intro-self-consistency}.
\end{theoremC}

\begin{theoremD}[Analytic finiteness]
Let $T$ be a forward expansive correspondence on a compact metric space, let
$\psi\in C(\GammaT,\R)$, and suppose that $(T,\psi)$ satisfies the
$C^\omega$ Legendre assumptions in \cref{def6.2} and has unique linear
equilibrium pairs.  If $F$ is
real analytic on a neighborhood of the convex hull of $\psi(\GammaT)$ and the
corresponding energy has abundance, then there are only finitely many nonlinear
equilibrium values and equilibrium pairs, and hence only finitely many
two-coordinate distributions of such pairs.
\end{theoremD}

Theorems~C and D are proved in
\cref{thm:legendre,thm:analytic-finiteness}.  The subsequent results apply these theorems to parameter-dependent energies
and finite-state correspondences.

For scalar quadratic energies, we also give a general continuous-transition
criterion.  When the response ratio $y/P'(y)$ is monotone and
$P'(y)=ay-by^{2k+1}+O(y^{2k+3})$, the critical inverse temperature is $1/a$,
the order parameter has exponent $1/(2k)$, and the singular part of the
nonlinear pressure has order $1+1/k$; see
\cref{thm:taylor-critical-exponent}.  The exponent $1/(2k)$ is stable as long as the first $k-1$
nonlinear odd coefficients remain zero and the coefficient of $y^{2k+1}$
remains negative.  A perturbation preserving the hypotheses of that theorem and creating an
earlier negative coefficient changes the exponent to the corresponding lower-order
value; see \cref{prop:response-monotonicity-stability,cor:taylor-stratification}.

The finite-state theory is illustrated by a completely explicit two-state
model.  Its linear pressure is $P(y)=\log(2\cosh y)$ and its mean-field
self-consistency equation is $m=\tanh(\beta m)$.  We prove, by comparing the
global variational values, that the model has one equilibrium pair for
$0\leq\beta\leq1$ and exactly two symmetry-related equilibrium pairs for
$\beta>1$; see \cref{prop:two-state-transition}.  For $\beta>1$, adding an
external field $h$ selects one of these phases.  The selected two-coordinate distribution jumps as
$h$ crosses zero, while the disfavored phase persists locally up to an explicit
spinodal threshold; see \cref{prop:field-transition}.

A multidimensional example is obtained from the three-state complete
correspondence with two state-occupation observables.  Its rotation set is a
two-dimensional simplex.  The associated mean-field Potts energy has a
temperature-driven first-order transition at $\beta_c=4\log2$, where one
disordered and three ordered two-coordinate distributions of equilibrium pairs coexist; see
\cref{prop:three-state-potts}.

The distinction between one-coordinate marginals and two-coordinate
distributions is visible already on a
two-state complete correspondence when the observable is
$\psi(s,t)=st$.  For inverse temperature $\beta>1$, the two equilibrium pairs
have the same uniform stationary measure but opposite correlations between successive states.
Thus the stationary measure alone does not distinguish the two equilibrium
pairs, whereas their two-coordinate distributions and transition kernels do;
see \cref{prop:edge-correlation}.

A constraint on the admissible transitions can also change the critical exponent.  For a
three-state correspondence with a primitive adjacency matrix and one forbidden
transition, the same two-coordinate observable has an explicit pressure
different from $\log(2\cosh y)$.  The
resulting nonlinear transition occurs at $\beta_c=\sqrt3$ and its order
parameter vanishes with exponent $1/4$; see
\cref{prop:constrained-correlation}.

The paper is organized as follows.  In \cref{sec:setting} we define nonlinear
pressure and stationary transition pairs.  In \cref{sec:markovization} we pass
to the orbit space and prove the Markov replacement statement.  The variational
principle and existence theorem are proved in \cref{sec:vp}.  Nonlinear Gibbs
ensembles are studied in \cref{sec:gibbs}.  Energies depending
on finitely many two-coordinate potentials and Theorems~C--D are treated in
\cref{sec:legendre}.  Finite-state examples and extensions are collected in
\cref{sec:finite}.

\section{Correspondences and nonlinear pressure}\label{sec:setting}

Throughout the paper, $(X,d)$ is a compact metric space, $\N=\{1,2,\ldots\}$,
and logarithms are natural. Write $\cF(X)$ for
the collection of nonempty closed subsets of $X$.

\begin{definition}
A \emph{correspondence} is a map $T:X\to\cF(X)$ for which the set
\[
  \GammaT=\{(x,y)\in X^2:y\in T(x)\}
\]
of admissible pairs is closed in $X^2$.  For $n\geq 1$, set
\[
  \cO_{n+1}(T)
  :=\{(x_0,\ldots,x_n)\in X^{n+1}:x_{j+1}\in T(x_j),\ 0\leq j<n\}.
\]
The one-sided orbit space and its shift are
\[
 \OmegaT:=\{(x_j)_{j\geq0}:x_{j+1}\in T(x_j)\},
 \qquad
 \sigma((x_j)_{j\geq0})=(x_{j+1})_{j\geq0}.
\]
\end{definition}

In particular, $\GammaT=\cO_2(T)$, with the notation of
\cite[Section~2]{li-li-zhang-correspondences}.  Each element $(x,y)\in\GammaT$
is a two-coordinate orbit: the state $y$ is an admissible successor of $x$.
A \emph{two-coordinate potential} (or observable) means a continuous function
on $\GammaT$; it need not depend on the first coordinate alone.
We use the coordinate projections
$\proj_1(x,y)=x$ and $\proj_2(x,y)=y$ on $X^2$ and on $\GammaT$.

The spaces $\cO_{n+1}(T)$ and $\OmegaT$ are compact.  On finite orbit spaces we
use
\[
 d_{n+1}(\boldsymbol x,\boldsymbol y)
 :=\max_{0\leq j\leq n}d(x_j,y_j).
\]
On $\OmegaT$ we fix the compatible metric
\begin{equation*}
 d_\omega(\boldsymbol x,\boldsymbol y)
 :=\sum_{j=0}^{\infty}2^{-j-1}
    \frac{d(x_j,y_j)}{1+d(x_j,y_j)}.
\end{equation*}
On $\GammaT$ we use the product metric
\[
 d_\Gamma((x_0,x_1),(y_0,y_1))
 :=\max\{d(x_0,y_0),d(x_1,y_1)\}.
\]
For a function $f:Y\to\R$ on a compact metric space $(Y,d_Y)$, write
$\|f\|_\infty:=\sup_{y\in Y}|f(y)|$ and let $\operatorname{Lip}(f)$ be the
least $L\geq0$ such that $|f(u)-f(v)|\leq Ld_Y(u,v)$ for all $u,v\in Y$,
with value $+\infty$ if no such $L$ exists.
The bounded-Lipschitz metric on
$\Prob(Y)$ is
\begin{equation*}
 d_{\mathrm{BL}}(\lambda,\lambda')
 :=\sup_{\substack{\|f\|_\infty\leq1\\ \operatorname{Lip}(f)\leq1}}
 \left|\int_Y f\,d\lambda-\int_Y f\,d\lambda'\right|.
\end{equation*}
The underlying metric space will always be clear from context.  This metric
induces the weak-star topology on $\Prob(Y)$.

For a Borel probability measure $\lambda$ on $Y$, its support is
\[
 \supp\lambda
 :=\{y\in Y:\lambda(U)>0\text{ for every open neighborhood }U\text{ of }y\}.
\]
For $\eta\in\Prob(\Prob(Y))$, the \emph{barycenter}
$\bar\lambda=\int\lambda\,d\eta(\lambda)\in\Prob(Y)$ is characterized by
\[
 \int_Y f\,d\bar\lambda
 =\int_{\Prob(Y)}\left(\int_Y f\,d\lambda\right)d\eta(\lambda)
 \qquad(f\in C(Y)).
\]
Thus a barycenter is an average of probability measures; if $\eta$ has finite
support, it is their convex combination.

For a finite probability vector $p=(p_i)_{i\in I}$, write
\[
 H(p):=-\sum_{i\in I}p_i\log p_i,
 \qquad 0\log0:=0.
\]
Let $\mathcal P_Y$ be the family of finite Borel partitions of $Y$, and set
$\operatorname{diam}\mathcal A:=\max_{A\in\mathcal A}
\sup_{u,v\in A}d_Y(u,v)$ for $\mathcal A\in\mathcal P_Y$.
For $\mathcal A\in\mathcal P_Y$, set
\[
 H_\lambda(\mathcal A)
 :=-\sum_{A\in\mathcal A}\lambda(A)\log\lambda(A).
\]
For partitions $\mathcal A$ and $\mathcal B$, their join
$\mathcal A\vee\mathcal B$ consists of the nonempty intersections
$A\cap B$, with $A\in\mathcal A$ and $B\in\mathcal B$.
For $\lambda\in\Prob(Y)$ and a sub-$\sigma$-algebra $\mathscr C$ of the
Borel $\sigma$-algebra of $Y$, let
$p_A:=\mathbb E_\lambda(\mathbf1_A\mid\mathscr C)$ denote the conditional
probability of $A\in\mathcal A$ given $\mathscr C$.  Define
\[
 H_\lambda(\mathcal A\mid\mathscr C)
 :=-\int_Y\sum_{A\in\mathcal A}p_A\log p_A\,d\lambda.
\]
Equivalently,
\[
 H_\lambda(\mathcal A\mid\mathscr C)
 =\inf\bigl\{H_\lambda(\mathcal A\mid\mathcal B):
     \mathcal B\in\mathcal P_Y,\ 
     \sigma(\mathcal B)\subseteq\mathscr C\bigr\},
\]
where $\sigma(\mathcal B)$ is the $\sigma$-algebra generated by the atoms of
$\mathcal B$, and conditioning on a finite partition means
\[
 H_\lambda(\mathcal A\mid\mathcal B)
 :=H_\lambda(\mathcal A\mid\sigma(\mathcal B))
 =H_\lambda(\mathcal A\vee\mathcal B)-H_\lambda(\mathcal B).
\]
The equivalence with the infimum formula follows from the martingale
convergence theorem for conditional expectations; see
\cite{petersen-ergodic-theory}.

If $S:Y\to Y$ is continuous and $\lambda\in\Inv(Y,S)$, the entropy rate of
$\mathcal A$ is
\[
 h_\lambda(S,\mathcal A)
 :=\lim_{n\to\infty}\frac1n
 H_\lambda\left(\bigvee_{j=0}^{n-1}S^{-j}\mathcal A\right).
\]
The limit exists by subadditivity.  The Kolmogorov--Sinai entropy is
\[
 h_\lambda(S):=\sup_{\mathcal A\in\mathcal P_Y}h_\lambda(S,\mathcal A).
\]
A partition is a \emph{one-sided generator} if its inverse images
$S^{-j}\mathcal A$, $j\geq0$, generate the Borel sigma-algebra, modulo
$\lambda$-null sets when working with a fixed measure $\lambda$.
For a finite signed Borel measure $\xi$ on $Y$, we use
the total-variation norm
\[
 \|\xi\|_{\mathrm{TV}}
 :=\sup_{\mathcal A\in\mathcal P_Y}
       \sum_{A\in\mathcal A}|\xi(A)|.
\]
A measurable space is \emph{standard Borel} if it admits a bijection onto
a Borel subset of a complete separable metric space such that the bijection
and its inverse are measurable. Compact metric spaces with their Borel
$\sigma$-algebras are standard Borel.

\begin{definition}
The correspondence $T$ is \emph{forward expansive} if there exists
$\eps_0>0$ such that, for every two distinct paths
$\boldsymbol x,\boldsymbol y\in\OmegaT$, one has
$d(x_j,y_j)>\eps_0$ for some $j\geq0$.
\end{definition}

A continuous map $S:Y\to Y$ of a compact metric space $(Y,d_Y)$ is
\emph{positively expansive} if there is $c>0$ such that, for every
$u\ne v$, some $j\geq0$ satisfies $d_Y(S^ju,S^jv)>c$.
Such a $c$ is called an expansive constant.
Forward expansivity of $T$ implies positive expansivity of the shift on
$\OmegaT$.  Indeed, if $\eps_0$ is an expansive constant for $T$ and
$\boldsymbol x\neq\boldsymbol y$, choose $j\geq0$ with
$d(x_j,y_j)>\eps_0$.  Then
\[
 d_\omega(\sigma^j\boldsymbol x,\sigma^j\boldsymbol y)
 >\frac{\eps_0}{2(1+\eps_0)}.
\]
Thus $\eps_0/(2(1+\eps_0))$ is an expansive constant for the shift
\cite[Proposition~6.3]{li-li-zhang-correspondences}.  In particular, its
entropy map is upper semicontinuous; the finite-generator proof needed below
is included in \cref{lem:entropy-usc-positive-expansive}.

Fix an energy $\cE\in C(\Prob(\GammaT),\R)$.  For
$\boldsymbol x\in\cO_{n+1}(T)$, define its \emph{empirical measure of
successive pairs} by
\begin{equation*}
  L_n(\boldsymbol x)
  :=\frac1n\sum_{j=0}^{n-1}\delta_{(x_j,x_{j+1})}.
\end{equation*}
For Borel sets $A,B\subseteq X$, this means
\[
 L_n(\boldsymbol x)(A\times B)
 =\frac1n\#\{0\leq j<n:x_j\in A,\ x_{j+1}\in B\}.
\]
Thus $L_n(\boldsymbol x)$ records the relative frequencies of successive
pairs, not merely the frequencies of individual states.  Its marginals satisfy
\[
 (\proj_1)_*L_n=\frac1n\sum_{j=0}^{n-1}\delta_{x_j},
 \qquad
 (\proj_2)_*L_n=\frac1n\sum_{j=1}^{n}\delta_{x_j},
\]
and hence
\[
 (\proj_1)_*L_n-(\proj_2)_*L_n
 =\frac{\delta_{x_0}-\delta_{x_n}}n.
\]
In particular, a finite empirical measure need not have equal marginals.
In \cref{sec:markovization} we identify $L_n$ as the two-coordinate projection
of the usual empirical measure on the orbit space.

\begin{definition}\label{def:nonlinear-pressure}
A subset $\cC\subset\cO_{n+1}(T)$ is \emph{$(n,\eps)$-separated} if
$d_{n+1}(\boldsymbol x,\boldsymbol y)>\eps$ whenever
$\boldsymbol x\ne\boldsymbol y$ belong to $\cC$.
For $\eps>0$ and $n\geq1$, let
\begin{equation*}
 Z_n(T,\cE,\eps)
 :=\sup_{\cC}
   \sum_{\boldsymbol x\in\cC}
   \exp\bigl(n\cE(L_n(\boldsymbol x))\bigr),
\end{equation*}
where the supremum ranges over all $(n,\eps)$-separated subsets
$\cC\subset\cO_{n+1}(T)$.  The upper and lower nonlinear topological
pressures are
\begin{align*}
 \overline\Pi_{\mathrm{top}}^{\cE}(T)
  &:=\lim_{\eps\to0^+}\limsup_{n\to\infty}
     \frac1n\log Z_n(T,\cE,\eps),\\
 \underline\Pi_{\mathrm{top}}^{\cE}(T)
  &:=\lim_{\eps\to0^+}\liminf_{n\to\infty}
     \frac1n\log Z_n(T,\cE,\eps).
\end{align*}
The outer limits are taken through positive $\eps$; they exist by
monotonicity in the separation scale. When the two pressures agree, their
common value is denoted by $\Ptop^{\cE}(T)$.
\end{definition}

\begin{remark}
The Carath\'eodory formulation of pressure on noncompact or selected orbit
sets, which underlies several later variants of nonlinear pressure, goes back
to Pesin and Pitskel'~\cite{pesin-pitskel-pressure-noncompact}.
If $\cE(\gamma)=\int_{\GammaT}\varphi\,d\gamma$ for a continuous two-coordinate potential $\varphi$, then $n\cE(L_n(\boldsymbol x))
 =\sum_{j=0}^{n-1}\varphi(x_j,x_{j+1}).$ Thus \cref{def:nonlinear-pressure} reduces to the linear topological pressure
of~\cite{li-li-zhang-correspondences}, which we denote by $P(T,\varphi)$.
\end{remark}

We next introduce the measure-theoretic objects. A Borel transition
probability kernel $Q$ on $X$ assigns a probability measure $Q_x$ to each
$x\in X$ so that $x\mapsto Q_x(A)$ is Borel for every Borel set $A$.
Let $\Ker(X;T)$ consist of these kernels satisfying
$\supp Q_x\subseteq T(x)$ for every $x\in X$. Write
\[
 \Inv(X,Q):=\{\mu\in\Prob(X):\mu Q=\mu\},
 \qquad
 (\mu Q)(A):=\int_X Q_x(A)\,d\mu(x).
\]
A pair $(\mu,Q)$ with $Q\in\Ker(X;T)$ and $\mu\in\Inv(X,Q)$ is called a
\emph{stationary transition pair}.  Its \emph{two-coordinate distribution}
is the Borel probability measure on $X^2$ given by
\begin{equation*}
 \gamma_{\mu,Q}(A)
 :=\int_X\int_X\mathbf1_A(x,y)\,dQ_x(y)\,d\mu(x),
 \qquad A\subseteq X^2\text{ Borel}.
\end{equation*}
Equivalently, for Borel sets $A,B\subseteq X$,
\[
 \gamma_{\mu,Q}(A\times B)=\int_A Q_x(B)\,d\mu(x).
\]
This is the joint distribution of two successive states of the stationary
Markov process with initial distribution $\mu$ and transition kernel $Q$.
In the notation of \cite[Definition~5.8]{li-li-zhang-correspondences}, it is
exactly $\mu Q^{[1]}$.  In contrast, $\mu Q$ is a probability measure on $X$.
More precisely,
\[
 (\proj_1)_*\gamma_{\mu,Q}=\mu,
 \qquad (\proj_2)_*\gamma_{\mu,Q}=\mu Q=\mu,
 \qquad \gamma_{\mu,Q}(\GammaT)=1.
\]
We identify $\gamma_{\mu,Q}$ with its restriction to $\GammaT$.  For every
$\varphi\in C(\GammaT,\R)$,
\[
 \int_{\GammaT}\varphi\,d\gamma_{\mu,Q}
 =\int_X\int_{T(x)}\varphi(x,y)\,dQ_x(y)\,d\mu(x).
\]
Thus integrating a potential against $\gamma_{\mu,Q}$ gives precisely the
potential term in the variational principle of
\cite[Theorem~A]{li-li-zhang-correspondences}.

\begin{definition}
The nonlinear measure pressure of a stationary transition pair is
\begin{equation}\label{eq:measure-pressure}
 \Pmeas_{\cE}(T;\mu,Q)
 :=h_\mu(Q)+\cE(\gamma_{\mu,Q}),
\end{equation}
where the kernel entropy is defined as follows.  For a finite Borel partition
$\mathcal A$ of $X$, let $\mathcal A^{\times n}$ be its $n$-fold product partition and
let $\mu Q^{[n-1]}$ be the probability measure on $X^n$ given, for $n\geq2$,
on rectangles by
\begin{align*}
 &\mu Q^{[n-1]}(A_0\times\cdots\times A_{n-1})\\
 &\quad:=\int_{A_0}\int_{A_1}\cdots\int_{A_{n-1}}
 dQ_{x_{n-2}}(x_{n-1})\cdots dQ_{x_0}(x_1)\,d\mu(x_0).
\end{align*}
For $n=1$ set $\mu Q^{[0]}:=\mu$.
Set
\begin{equation*}
 h_\mu(Q,\mathcal A)
 :=\lim_{n\to\infty}\frac1n
   H_{\mu Q^{[n-1]}}(\mathcal A^{\times n}),
 \qquad
 h_\mu(Q):=\sup_{\mathcal A\in\mathcal P_X}h_\mu(Q,\mathcal A).
\end{equation*}
The limit exists by the usual subadditivity argument for the stationary path
process.  This is the kernel entropy introduced in
\cite[Definitions~5.12--5.13]{li-li-zhang-correspondences}.  A stationary transition pair maximizing \eqref{eq:measure-pressure} over
$Q\in\Ker(X;T)$ and $\mu\in\Inv(X,Q)$ is called a
\emph{nonlinear equilibrium pair}.
\end{definition}

Kernels that agree $\mu$-almost everywhere determine the same two-coordinate distribution, path
measure, entropy, and pressure.  Equilibrium kernels will always be understood
modulo this equivalence.

\section{Passage to the orbit space and Markov replacement}\label{sec:markovization}

Let $p_2:\OmegaT\to\GammaT$ be the projection
$p_2((x_j)_{j\geq0})=(x_0,x_1)$.  Lift the energy to
$\Prob(\OmegaT)$ by $\widehat{\cE}(\nu):=\cE((p_2)_*\nu)$.

The lifted energy is continuous. For $\boldsymbol x\in\OmegaT$, define
the empirical measure on the path space by
\[
 \Delta_n(\boldsymbol x):=\frac1n\sum_{j=0}^{n-1}\delta_{\sigma^j\boldsymbol x}.
\]
This is the construction of
\cite[Section~1.2]{buzzi-kloeckner-leplaideur-nonlinear} for
$(\OmegaT,\sigma)$, and its two-coordinate projection is
\begin{equation}\label{eq:empirical-compatibility}
 (p_2)_*\Delta_n(\boldsymbol x)=L_n(x_0,\ldots,x_n).
\end{equation}

For $n\geq1$, write
\[
 d_{\omega,n}(\boldsymbol x,\boldsymbol y)
 :=\max_{0\leq k<n}d_\omega(\sigma^k\boldsymbol x,
                              \sigma^k\boldsymbol y)
\]
and let
\begin{equation*}
 \widehat Z_n(\widehat\cE,\delta)
 :=\sup_{\widehat\cC}
   \sum_{\boldsymbol x\in\widehat\cC}
   \exp\bigl(n\widehat\cE(\Delta_n(\boldsymbol x))\bigr),
\end{equation*}
where $\widehat\cC$ ranges over the $(n,\delta)$-separated subsets of
$(\OmegaT,d_{\omega,n})$.
The upper and lower nonlinear pressures of $(\OmegaT,\sigma,\widehat\cE)$
are defined by the same two limits as in \cref{def:nonlinear-pressure},
with $\widehat Z_n(\widehat\cE,\delta)$ in place of $Z_n(T,\cE,\eps)$.

Using the bounded-Lipschitz metric on $\Prob(\GammaT)$, define the modulus
of continuity
\[
 \omega_{\cE}(t):=\sup\bigl\{|\cE(\gamma)-\cE(\gamma')|:
 d_{\mathrm{BL}}(\gamma,\gamma')\leq t\bigr\},\qquad t\geq0,
\]
where $\gamma,\gamma'\in\Prob(\GammaT)$.
Compactness gives $\lim_{t\to0^+}\omega_{\cE}(t)=0$.

\begin{lemma}\label{lem:edge-lift}
Let $T$ be a correspondence on a compact metric space and let
$\cE\in C(\Prob(\GammaT),\R)$.
Let $M_{\cE}:=\|\cE\|_\infty$.
\begin{enumerate}[label=\textup{(\roman*)}]
 \item For every $\eps>0$ and $n\geq1$,
 \begin{equation}\label{eq:edge-to-path-estimate}
  Z_n(T,\cE,\eps)
  \leq
  \widehat Z_n\left(\widehat\cE,
       \frac{\eps}{4(1+\eps)}\right).
 \end{equation}
 \item Fix $0<\delta<1$ and choose an integer $r=r(\delta)\geq1$ so that
 $2^{-r-1}<\delta/2$.  Set $m=n+r$ and $a=\delta/4$.  Then
 \begin{equation}\label{eq:path-to-edge-estimate}
  \widehat Z_n(\widehat\cE,\delta)
  \leq
  \exp\left(n\omega_{\cE}\left(\frac{2r}{n+r}\right)
               +rM_{\cE}\right)
  Z_{n+r}(T,\cE,a).
 \end{equation}
\end{enumerate}
Consequently, the upper and lower nonlinear pressures in
\cref{def:nonlinear-pressure} coincide, respectively, with the upper and lower
nonlinear pressures of $(\OmegaT,\sigma,\widehat\cE)$.
\end{lemma}

\begin{proof}
For part~\textup{(i)}, we use the finite-orbit extension argument of
\cite[Theorem~4.9, Step~1]{li-li-zhang-correspondences}.
For part~\textup{(ii)}, we compare prefixes of different lengths and estimate
the change in their empirical measures.

Every finite admissible orbit extends to an element of $\OmegaT$, since the
values of $T$ are nonempty.  Let $\cC\subset\cO_{n+1}(T)$ be
$(n,\eps)$-separated and choose one extension of every element.  For two
distinct extensions there is $0\leq j\leq n$ such that
$d(x_j,y_j)>\eps$.  If $j<n$, then
\[
 d_\omega(\sigma^j\boldsymbol x,\sigma^j\boldsymbol y)
 >\frac{\eps}{2(1+\eps)}.
\]
If $j=n$, the same coordinate is visible with coefficient $1/4$ from time
$n-1$, giving the lower bound $\eps/(4(1+\eps))$.  The extensions therefore
form an $(n,\eps/(4(1+\eps)))$-separated set.  Their weights agree exactly, which proves
\eqref{eq:edge-to-path-estimate}.

Conversely, let $\widehat\cC\subset\OmegaT$ be
$(n,\delta)$-separated.  If
\[
 d(x_{k+j},y_{k+j})\leq a=\delta/4
 \quad\text{for }0\leq j\leq r,
\]
then
\[
 d_\omega(\sigma^k\boldsymbol x,\sigma^k\boldsymbol y)
 \leq \frac{\delta}{4}+2^{-r-1}<\frac{3\delta}{4}.
\]
Hence Bowen separation supplies a coordinate among the first $n+r$
coordinates at which the distance is larger than $a$.  The prefixes of length
$n+r+1$ are therefore distinct and $(n+r,a)$-separated.

Let $m=n+r$.  For every $\boldsymbol x\in\widehat\cC$,
\[
 d_{\mathrm{BL}}\bigl(L_n(x_0,\ldots,x_n),
                       L_m(x_0,\ldots,x_m)\bigr)
 \leq\frac{2r}{n+r}.
\]
It follows that
\begin{align*}
 n\cE(L_n)
 &\leq n\cE(L_m)
       +n\omega_{\cE}\left(\frac{2r}{n+r}\right)\\
 &\leq m\cE(L_m)
       +n\omega_{\cE}\left(\frac{2r}{n+r}\right)+rM_{\cE}.
\end{align*}
Summing the exponentials proves \eqref{eq:path-to-edge-estimate}.

For fixed separation scales, $r$ is independent of $n$, and the logarithm
of the multiplicative factor in \eqref{eq:path-to-edge-estimate}, divided by
$n$, tends to zero:
\[
 \lim_{n\to\infty}
 \left[\omega_{\cE}\left(\frac{2r}{n+r}\right)+\frac{rM_{\cE}}n\right]=0.
\]
The index shift from
$n$ to $n+r$ does not change either normalized limsup or normalized liminf.
Taking the limits in \eqref{eq:edge-to-path-estimate} and
\eqref{eq:path-to-edge-estimate}, and then letting the separation scales tend
to zero, proves the final assertion.
\end{proof}

The key step is that an arbitrary stationary path measure can be replaced by a
Markov path measure without changing the energy.

\begin{proposition}[Markov replacement]\label{prop:markovization}
Let $T$ be a forward expansive correspondence on a compact metric space.
For every
$\nu\in\Inv(\OmegaT,\sigma)$ there exist
$Q\in\Ker(X;T)$ and $\mu\in\Inv(X,Q)$ such that
\begin{equation}\label{eq:markovization-marginal}
  (p_2)_*\nu=\gamma_{\mu,Q}
\end{equation}
and
\begin{equation}\label{eq:markovization-entropy}
  h_\nu(\sigma)\leq h_\mu(Q).
\end{equation}
Conversely, every stationary pair $(\mu,Q)$ induces a shift-invariant Markov
measure $\nu_{\mu,Q}$ on $\OmegaT$ satisfying
\begin{equation}\label{eq:markov-converse}
 (p_2)_*\nu_{\mu,Q}=\gamma_{\mu,Q},
 \qquad
 h_{\nu_{\mu,Q}}(\sigma)=h_\mu(Q).
\end{equation}
\end{proposition}

\begin{proof}
Using disintegration of the two-coordinate marginal and conditional entropy,
as in the proofs of
\cite[Lemma~6.17 and Proposition~6.19]{li-li-zhang-correspondences}, let
$\gamma=(p_2)_*\nu$ and let $\mu$ be its first marginal.  For every Borel
$A\subseteq X$, shift invariance gives
\[
 \gamma(X\times A)
 =\nu\{x_1\in A\}
 =\nu\{x_0\in A\}
 =\gamma(A\times X),
\]
so the two marginals of $\gamma$ are both $\mu$.  Since $X$ is a compact
metric space, disintegration with respect to the first coordinate gives a
Borel kernel $Q^0$ such that
\[
 \gamma(dx,dy)=\mu(dx)Q^0_x(dy),
 \qquad Q^0_x(T(x))=1
 \quad\text{for $\mu$-almost every $x$}.
\]
The Kuratowski--Ryll-Nardzewski selection theorem
\cite[Theorem~12.13]{kechris-descriptive-set-theory} gives a Borel map
$s:X\to X$ with $s(x)\in T(x)$.  The exceptional set on which
$Q_x^0(T(x))<1$ is Borel, because integration of the Borel function
$\mathbf1_{\GammaT}(x,y)$ against the kernel $Q_x^0$ is Borel.  On this
$\mu$-null set
where the support assertion for $Q^0$ may fail, replace $Q^0_x$ by
$\delta_{s(x)}$.  The resulting Borel kernel $Q$ lies in $\Ker(X;T)$ and still
satisfies
\[
 \gamma(dx,dy)=\mu(dx)Q_x(dy).
\]
Taking the second marginal now yields
\[
 (\mu Q)(A)=\gamma(X\times A)=\mu(A),
\]
and hence $\mu\in\Inv(X,Q)$ and
\eqref{eq:markovization-marginal} holds.

The needed conditional-entropy calculation is as follows.
If $c>0$ is a forward expansive constant, then any two distinct points in $T^{-1}(x)$
are $c$-separated: append the same future path beginning at $x$ and apply
forward expansivity.  Here
$T^{-1}(x):=\{y\in X:x\in T(y)\}$.  Compactness of $X$ therefore gives a uniform finite
bound on the cardinalities of the inverse fibres.  Choose a finite Borel
partition $\mathcal A$ of $X$ whose atoms have diameter smaller than $c$.

Let $\bar\nu$ be the stationary probability measure on the two-sided
admissible path space
$\{(x_j)_{j\in\mathbb Z}\in X^{\mathbb Z}:x_{j+1}\in T(x_j)
\text{ for all }j\in\mathbb Z\}$ whose restriction to the coordinates $j\geq0$ is
$\nu$; its finite-coordinate distributions are determined by stationarity.
Write $(X_j)_{j\in\mathbb Z}$ for its coordinate process, and let
$\mathcal A_j$ be the partition according to the $\mathcal A$-atom containing
$X_j$. The $\mathcal A$-name of a path is the sequence of partition atoms
containing its coordinates. Equality of all nonnegative $\mathcal A$-names keeps two admissible
paths at distance less than $c$ at every time, so forward expansivity makes
the coordinate partition a one-sided generator.  Hence the finite-generator
entropy identity, followed by continuity of conditional entropy along the
increasing future sigma-algebras, gives
\begin{equation}\label{eq:conditional-future-entropy}
 \begin{split}
 h_\nu(\sigma)
 &=\lim_{m\to\infty}
   H_{\bar\nu}\left(\mathcal A_{-1}\,\middle|\,
      \bigvee_{j=0}^{m}\mathcal A_j\right)\\
 &=H_{\bar\nu}\left(\mathcal A_{-1}\,\middle|\,
      \sigma(X_j:j\geq0)\right).
 \end{split}
\end{equation}
The last sigma-algebra is indeed the whole future sigma-algebra: the
nonnegative $\mathcal A$-name determines the future path, and an injective
Borel map between standard Borel spaces has a Borel inverse on its image.

Denote by $P_{\boldsymbol x}\in\Prob(X)$ the conditional probability
measure of $X_{-1}$ given the future $\boldsymbol x=(X_0,X_1,\ldots)$, by
$R_x\in\Prob(X)$ the conditional probability measure of $X_{-1}$ given
$X_0=x$, and by $S_x\in\Prob(\OmegaT)$ that of the future given $X_0=x$.
These regular conditional probabilities exist because the spaces are
standard Borel. Almost surely, $P_{\boldsymbol x}$ and $R_x$ are supported
on the finite sets $T^{-1}(x_0)$ and $T^{-1}(x)$, respectively.
For such a measure $\lambda$, write
$H(\lambda):=-\sum_{u\in\supp\lambda}\lambda(\{u\})\log\lambda(\{u\})$.
Each atom of $\mathcal A$ contains at most one point of an inverse fibre,
so \eqref{eq:conditional-future-entropy} becomes
\begin{equation}\label{eq:path-backward-entropy}
 h_\nu(\sigma)
 =\int_{\OmegaT}H(P_{\boldsymbol x})\,d\nu(\boldsymbol x).
\end{equation}
The tower property also gives
\begin{equation}\label{eq:backward-kernel-mixture}
 R_x=\int_{\OmegaT}P_{\boldsymbol y}\,dS_x(\boldsymbol y)
 \quad\text{for $\mu$-almost every $x$}.
\end{equation}

Conversely, given a stationary pair $(\mu,Q)$, Kolmogorov extension defines a
probability measure $\nu_{\mu,Q}$ by the cylinder formula
\begin{align*}
 &\nu_{\mu,Q}(A_0\times\cdots\times A_n\times X^{\N})\\
 &\qquad:=\int_{A_0}\int_{A_1}\cdots\int_{A_n}
 dQ_{x_{n-1}}(x_n)\cdots dQ_{x_0}(x_1)\,d\mu(x_0).
\end{align*}
The support property of $Q$ gives $\nu_{\mu,Q}(\OmegaT)=1$, while
$\mu Q=\mu$ gives shift invariance.  Taking $n=1$ in the cylinder formula
gives $(p_2)_*\nu_{\mu,Q}=\gamma_{\mu,Q}$.

For the Markov path measure, kernel entropy agrees with shift entropy.  Let
$\pi_0(\boldsymbol x)=x_0$.  For a finite Borel partition
$\mathcal B\in\mathcal P_X$, write
$\widehat{\mathcal B}:=\pi_0^{-1}\mathcal B$.  Since the $n$-coordinate marginal
of $\nu_{\mu,Q}$ is $\mu Q^{[n-1]}$, the definition of kernel entropy gives
\begin{align*}
 h_\mu(Q,\mathcal B)
 &=\lim_{n\to\infty}\frac1n
   H_{\mu Q^{[n-1]}}(\mathcal B^{\times n})\\
 &=\lim_{n\to\infty}\frac1n
   H_{\nu_{\mu,Q}}
    \left(\bigvee_{j=0}^{n-1}
       \sigma^{-j}\widehat{\mathcal B}\right)\\
 &=h_{\nu_{\mu,Q}}(\sigma,\widehat{\mathcal B}).
\end{align*}
Every finite partition in the cylinder algebra is refined by
$\bigvee_{j=0}^{m-1}\sigma^{-j}\widehat{\mathcal B}$ for some $m$ and some
finite $\mathcal B$.  The cylinder algebra generates the Borel sigma-algebra
of $\OmegaT$, so the generating-algebra form of the Kolmogorov--Sinai theorem
and the identity
\[
 h_{\nu_{\mu,Q}}\left(
   \sigma,\bigvee_{j=0}^{m-1}\sigma^{-j}\widehat{\mathcal B}\right)
 =h_{\nu_{\mu,Q}}(\sigma,\widehat{\mathcal B})
\]
give
\begin{equation}\label{eq:kernel-path-entropy}
 h_\mu(Q)
 =\sup_{\mathcal B\in\mathcal P_X}h_\mu(Q,\mathcal B)
 =h_{\nu_{\mu,Q}}(\sigma).
\end{equation}
This is the finite-coordinate identity underlying
\cite[Lemma~5.16, Theorem~5.15 and Remark~6.15]{li-li-zhang-correspondences}.

The two-coordinate distributions of $\nu$ and $\nu_{\mu,Q}$ agree.  Therefore $R_x$
is also the conditional probability measure of the preceding state given the present state
under $\nu_{\mu,Q}$.  The cylinder formula shows that a stationary Markov
chain has conditionally independent past and future given the present:
conditioning $X_{-1}$ on $(X_0,\ldots,X_m)$ gives $R_{X_0}$ for every $m$,
and a monotone-class argument gives the same conditional probability measure after observing
the entire future.  Applying \eqref{eq:path-backward-entropy} to
$\nu_{\mu,Q}$ thus yields
\begin{equation}\label{eq:markov-backward-entropy}
 h_{\nu_{\mu,Q}}(\sigma)=\int_X H(R_x)\,d\mu(x).
\end{equation}
By concavity of Shannon entropy and
\eqref{eq:backward-kernel-mixture},
$\int_{\OmegaT}H(P_{\boldsymbol y})\,dS_x(\boldsymbol y)\leq H(R_x)$
for $\mu$-almost every $x$.
After integration in $x$, equations
\eqref{eq:path-backward-entropy}, \eqref{eq:markov-backward-entropy}, and
\eqref{eq:kernel-path-entropy} give
$h_\nu(\sigma)
 \leq h_{\nu_{\mu,Q}}(\sigma)
 =h_\mu(Q),$ which proves \eqref{eq:markovization-entropy}.  The cylinder construction,
the two-coordinate identity, and \eqref{eq:kernel-path-entropy} prove
\eqref{eq:markov-converse}.
\end{proof}

\begin{corollary}\label{cor:sup-markov}
Let $T$ be a forward expansive correspondence on a compact metric space and
let $\cE\in C(\Prob(\GammaT),\R)$.  Then
\[
\sup_{\nu\in\Inv(\OmegaT,\sigma)}
   \bigl\{h_\nu(\sigma)+\widehat\cE(\nu)\bigr\}=\sup_{\substack{Q\in\Ker(X;T)\\\mu\in\Inv(X,Q)}}
   \bigl\{h_\mu(Q)+\cE(\gamma_{\mu,Q})\bigr\}.
\]
\end{corollary}

\begin{proof}
Apply \cref{prop:markovization}.  The Markov replacement preserves the
two-coordinate marginal and hence the energy, while it does not decrease
entropy.  The converse construction shows that every stationary transition
pair occurs among the path measures.
\end{proof}

\section{Variational principle and equilibrium pairs}\label{sec:vp}

\begin{definition}\label{def:abundance}
We say that $(T,\cE)$ has an \emph{abundance of ergodic path measures} if, for
every $\nu\in\Inv(\OmegaT,\sigma)$ and every $\eta>0$, there exists an ergodic
$\rho\in\Inv(\OmegaT,\sigma)$ such that $h_\rho(\sigma)+\widehat\cE(\rho)
 >h_\nu(\sigma)+\widehat\cE(\nu)-\eta.$

\end{definition}

This is \cite[Definition~1.1]{buzzi-kloeckner-leplaideur-nonlinear} for the
system $(\OmegaT,\sigma)$ with energy $\widehat\cE$.

The ergodic measures of $(\OmegaT,\sigma)$ are called \emph{entropy dense}
if, for every $\nu\in\Inv(\OmegaT,\sigma)$, every weak-star neighborhood
$\mathcal U$ of $\nu$, and every $\eta>0$, there is an ergodic
$\rho\in\mathcal U$ such that
$h_\rho(\sigma)>h_\nu(\sigma)-\eta$.

\begin{proposition}\label{prop:abundance-criteria}
Let $T$ be a forward expansive correspondence on a compact metric space and
let $\cE\in C(\Prob(\GammaT),\R)$.
Each of the following conditions implies abundance of ergodic path measures.
\begin{enumerate}[label=\textup{(\roman*)}]
 \item The ergodic measures of $(\OmegaT,\sigma)$ are entropy dense.
 \item The energy $\cE$ is convex on $\Prob(\GammaT)$.
\end{enumerate}
In particular, condition \textup{(i)} applies whenever a specification theorem
for $T$ yields entropy density for its forward expansive orbit shift.
\end{proposition}

\begin{proof}
Under \textup{(i)}, given $\nu$ and $\eta>0$, entropy density supplies an
ergodic $\rho$ arbitrarily close to $\nu$ with
$h_\rho(\sigma)>h_\nu(\sigma)-\eta/2$.  Choose it close enough that continuity
of $\widehat\cE$ also gives
$\widehat\cE(\rho)>\widehat\cE(\nu)-\eta/2$.  This is precisely abundance.

For \textup{(ii)}, let $\nu=\int\rho\,d\tau(\rho)$ be the ergodic decomposition:
$\tau$ is a probability measure on $\Inv(\OmegaT,\sigma)$, concentrated on
the ergodic measures, with barycenter $\nu$.
Entropy satisfies the ergodic-decomposition formula, and
\[
 (p_2)_*\nu=\int (p_2)_*\rho\,d\tau(\rho).
\]
Convexity gives
\[
 \widehat\cE(\nu)
 \leq\int\widehat\cE(\rho)\,d\tau(\rho).
\]
The affinity of entropy gives
\[
 h_\nu(\sigma)=\int h_\rho(\sigma)\,d\tau(\rho),
\]
and consequently
\[
 h_\nu(\sigma)+\widehat\cE(\nu)
 \leq\int
   \bigl(h_\rho(\sigma)+\widehat\cE(\rho)\bigr)\,d\tau(\rho).
\]
For every $\eta>0$, the set of ergodic components whose nonlinear pressure is
greater than $h_\nu(\sigma)+\widehat\cE(\nu)-\eta$ must therefore have positive
$\tau$-measure.  Any component in this set proves abundance.
\end{proof}

We use the specification property of
\cite[Definition~7.1]{li-li-zhang-correspondences}: for every $\eps>0$
there is $M\in\N$ such that any finite collection of admissible orbit
segments can be traced, in the prescribed order, by one path with error
less than $\eps$, provided the prescribed gaps exceed $M$.
More precisely, for segments $(x^{(i)}_j)_{j=0}^{n_i}$, $1\leq i\leq k$,
and integers $0=a_1<\cdots<a_k$ with
$a_{i+1}-(a_i+n_i+1)>M$, there is $\boldsymbol y\in\OmegaT$ such that
\[
 d(y_{a_i+j},x^{(i)}_j)<\eps
 \qquad(1\leq i\leq k,\ 0\leq j\leq n_i).
\]

\begin{corollary}\label{cor:specification-abundance}
Let $T$ be a forward expansive correspondence on a compact metric space with
the specification property of
\cite[Definition~7.1]{li-li-zhang-correspondences}.  Then $(T,\cE)$ has an
abundance of ergodic path measures for every
$\cE\in C(\Prob(\GammaT),\R)$.
\end{corollary}
\begin{proof}
By \cite[Propositions~6.3 and~7.3]{li-li-zhang-correspondences}, the orbit
shift is positively expansive and has specification. Specification implies
the approximate product property used in
\cite[Theorem~2.1]{pfister-sullivan-approximate-product}: arbitrarily long
orbit blocks can be traced in sequence with arbitrarily small relative gaps
and an arbitrarily small proportion of untraced times. Indeed, for a fixed
accuracy, specification traces every time in each block and permits a fixed
gap length; the ratio of this length to the block length tends to zero.
Compactness extends tracing from finitely many blocks to an infinite sequence.
The cited theorem therefore makes the ergodic path measures entropy dense.
Apply \cref{prop:abundance-criteria}(i).
\end{proof}

For a continuous map $S:Y\to Y$ of a compact metric space $(Y,d_Y)$,
write
\[
 d_n^S(u,v):=\max_{0\leq j<n}d_Y(S^ju,S^jv).
\]
Here $(n,\delta)$-separation means pairwise $d_n^S$-distance greater than
$\delta$. We omit the superscript $S$ when the map is fixed.

\begin{lemma}[Entropy bound for empirical-measure distributions]
\label{lem:empirical-distribution-entropy}
Let $S:Y\to Y$ be a positively expansive continuous map of a compact metric
space, and let $0<\delta$ be smaller than an expansive constant.  For
$n_k\to\infty$, let $E_k\subset Y$ be $(n_k,\delta)$-separated and let
$(p_{k,x})_{x\in E_k}$ be a probability vector.  Set
\[
 \Xi_k:=\sum_{x\in E_k}p_{k,x}\,\delta_{\Delta_{n_k}(x)},
 \qquad
 \Delta_n(x):=\frac1n\sum_{j=0}^{n-1}\delta_{S^j x}.
\]
If $\Xi_k$ converges weakly to $\Xi\in\Prob(\Prob(Y))$, then $\Xi$ is
supported on $\Inv(Y,S)$ and
\begin{equation}\label{eq:empirical-distribution-entropy-bound}
 \limsup_{k\to\infty}\frac1{n_k}
   H((p_{k,x})_{x\in E_k})
 \leq\int_{\Inv(Y,S)}h_\rho(S)\,d\Xi(\rho).
\end{equation}
\end{lemma}

\begin{proof}
Choose a countable uniformly dense family $(u_\ell)_{\ell\geq1}$ in $C(Y)$.
For every $u\in C(Y)$ and every $x\in Y$,
\begin{equation*}
 \left|\int(u-u\circ S)\,d\Delta_n(x)\right|
 =\frac{|u(x)-u(S^n x)|}{n}
 \leq\frac{2\|u\|_\infty}{n}.
\end{equation*}
For fixed $\ell$, the map
\[
 F_\ell(\lambda):=
 \left|\int(u_\ell-u_\ell\circ S)\,d\lambda\right|
\]
is continuous on $\Prob(Y)$, and
$\int F_\ell\,d\Xi_k\leq2\|u_\ell\|_\infty/n_k$.
Weak convergence therefore gives $\int F_\ell\,d\Xi=0$.  Intersecting the
resulting full-measure sets over $\ell$ and using density shows that
$\Xi(\Inv(Y,S))=1$.  Since $\Inv(Y,S)$ is closed, the topological support of
$\Xi$ is contained in $\Inv(Y,S)$.

Let
\[
 m_k:=\sum_{x\in E_k}p_{k,x}\delta_x,
 \qquad
 \bar\nu_k:=\int\rho\,d\Xi_k(\rho)
 =\frac1{n_k}\sum_{j=0}^{n_k-1}S_*^j m_k.
\]
The barycenter map is continuous, so
$\bar\nu_k\to\bar\nu:=\int\rho\,d\Xi(\rho)$.  The preceding paragraph implies
that $\bar\nu$ is $S$-invariant.  Choose a finite Borel partition
$\mathcal A$ of diameter smaller than $\delta$, with
$\bar\nu(\partial A)=0$ for every $A\in\mathcal A$; such a partition is
obtained from a finite cover by sufficiently small balls whose boundary
spheres are $\bar\nu$-null.  Write
\[
 \mathcal A^r:=\bigvee_{j=0}^{r-1}S^{-j}\mathcal A.
\]
Distinct points of $E_k$ have distinct $\mathcal A$-names of length $n_k$,
since two points with the same name would have Bowen distance smaller than
$\delta$.  Hence
\[
 H((p_{k,x})_{x\in E_k})
 =H_{m_k}(\mathcal A^{n_k}).
\]
The following calculation adapts the Misiurewicz block argument in
\cite[Lemma~2.6]{buzzi-kloeckner-leplaideur-nonlinear}.  Specifically, for a
fixed block length $q$ and $n_k\geq2q$, split the interval
$\{0,\ldots,n_k-1\}$ in each of the $q$ possible offsets into blocks of length
$q$ and two remainders of total length at most $2q$.  Subadditivity of
partition entropy, followed by summation over the offsets, gives
\[
 \frac{q}{n_k}H_{m_k}(\mathcal A^{n_k})
 \leq
 \frac1{n_k}\sum_{j=0}^{n_k-1}H_{S_*^j m_k}(\mathcal A^q)
 +\frac{2q^2\log\#\mathcal A}{n_k}.
\]
Concavity of $\lambda\mapsto H_\lambda(\mathcal A^q)$ and the definition of
$\bar\nu_k$ therefore imply
\begin{equation}\label{eq:misiurewicz-block-bound}
 \frac1{n_k}H_{m_k}(\mathcal A^{n_k})
 \leq\frac1qH_{\bar\nu_k}(\mathcal A^q)
 +\frac{2q\log\#\mathcal A}{n_k}.
\end{equation}
Writing $\partial\mathcal A:=\bigcup_{A\in\mathcal A}\partial A$,
invariance of $\bar\nu$ makes every inverse image of $\partial\mathcal A$
null, and
\[
 \partial\mathcal A^q
 \subseteq\bigcup_{j=0}^{q-1}S^{-j}\partial\mathcal A.
\]
Hence $\bar\nu(\partial\mathcal A^q)=0$.  Weak convergence then gives
$H_{\bar\nu_k}(\mathcal A^q)\to H_{\bar\nu}(\mathcal A^q)$.  Letting first
$k\to\infty$ in \eqref{eq:misiurewicz-block-bound} and then $q\to\infty$
yields
\[
 \limsup_{k\to\infty}\frac1{n_k}H((p_{k,x})_{x\in E_k})
 \leq h_{\bar\nu}(S,\mathcal A).
\]
Since $\operatorname{diam}\mathcal A<\delta$ and $\delta$ is expansive,
$\mathcal A$ is a one-sided generating partition: equal nonnegative
$\mathcal A$-itineraries force two points to coincide.  The
Kolmogorov--Sinai theorem therefore gives
$h_{\bar\nu}(S,\mathcal A)=h_{\bar\nu}(S)$.

The last step uses the entropy-decomposition formula described in
\cite{walters-ergodic-theory}.  Specifically, let $\vartheta_\rho$ denote the
ergodic decomposition of
$\rho\in\Inv(Y,S)$, viewed as a probability measure on $\Inv(Y,S)$
concentrated on the ergodic measures and having barycenter $\rho$.
The probability
$\Theta:=\int\vartheta_\rho\,d\Xi(\rho)$ is supported on the ergodic measures
and has barycenter $\bar\nu$; uniqueness of ergodic decomposition makes it the
ergodic decomposition of $\bar\nu$.  Applying the entropy decomposition
formula twice and Tonelli's theorem gives
\[
 \begin{split}
 h_{\bar\nu}(S)
 &=\int h_\omega(S)\,d\Theta(\omega)\\
 &=\int\left(\int h_\omega(S)\,d\vartheta_\rho(\omega)\right)d\Xi(\rho)
 =\int h_\rho(S)\,d\Xi(\rho).
 \end{split}
\]
This proves \eqref{eq:empirical-distribution-entropy-bound}.
\end{proof}

\begin{lemma}\label{lem:fixed-scale-stabilization}
Let $S:Y\to Y$ be a positively expansive continuous map of a compact metric
space $(Y,d)$, let $c>0$ be an expansive constant, and let
$\mathcal G\in C(\Prob(Y),\R)$. Suppose that the upper and lower nonlinear
pressures of $(Y,S,\mathcal G)$, defined by the two limits in
\cref{def:nonlinear-pressure} using the partition sums below, agree with
common value $P$. Then, for every
$0<\delta<c/2$,
\begin{equation}\label{eq:abstract-fixed-scale}
 P=\lim_{n\to\infty}\frac1n\log Z_n(S,\mathcal G,\delta).
\end{equation}
Here $Z_n(S,\mathcal G,r)$ denotes the
separated-set partition sum
\[
 Z_n(S,\mathcal G,r)
 :=\sup_E\sum_{x\in E}
   \exp\bigl(n\mathcal G(\Delta_n(x))\bigr),
\]
where $E$ ranges over the $(n,r)$-separated subsets of $Y$, and
\[
 \Delta_n(x):=\frac1n\sum_{j=0}^{n-1}\delta_{S^j x}.
\]
\end{lemma}

\begin{proof}
Fix $0<\eps<\delta<c/2$ and $\gamma>0$.  Positive expansivity and compactness
give the following uniform form of expansivity: for every $a>0$ and every
$b<c$, there is $N=N(a,b)$ such that
\begin{equation}\label{eq:uniform-positive-expansivity}
 \max_{0\leq j<N}d(S^j u,S^j v)\leq b
 \quad\Longrightarrow\quad d(u,v)<a.
\end{equation}
Indeed, otherwise compactness would produce distinct $u,v$ whose entire
forward orbits remain at distance at most $b<c$.

Let $E$ be any $(n,\eps)$-separated set and choose a subset
$\widehat E\subset E$ maximal $(n,\delta)$-separated.  Assign to every
$x\in E$ a point $\pi(x)\in\widehat E$ with
$d_n(x,\pi(x))\leq\delta$.  Apply
\eqref{eq:uniform-positive-expansivity} with $b=\delta$.  Given $a>0$, for all
large $n$ we obtain
\[
 d(S^j x,S^j\pi(x))<a\quad(0\leq j\leq n-N),
\]
and hence
\[
 d_{\mathrm{BL}}(\Delta_n(x),\Delta_n(\pi(x)))
 \leq a+\frac{2N}{n}.
\]
Choose $a$ and then $n$ so that uniform continuity of $\mathcal G$ makes the
corresponding energy difference smaller than $\gamma$.

The fibres of $\pi$ have cardinality bounded independently of $n$.  In fact,
if $x,z$ lie in the same fibre, then $d_n(x,z)\leq2\delta<c$.
Set $N':=N(\eps,2\delta)$. Using
\eqref{eq:uniform-positive-expansivity} with $b=2\delta$ and $a=\eps$
shows that $x$ and $z$ cannot be separated before the last $N'$ iterates.
Thus the images of a fibre under $S^{n-N'}$ form an
$(N',\eps)$-separated set.  Compactness bounds its cardinality by a constant
$M=M(\eps,N')$.

Consequently, for all large $n$,
\[
 \sum_{x\in E}e^{n\mathcal G(\Delta_n(x))}
 \leq M e^{n\gamma}
 \sum_{y\in\widehat E}e^{n\mathcal G(\Delta_n(y))},
\]
and therefore
\begin{equation}\label{eq:fixed-scale-comparison}
 Z_n(S,\mathcal G,\eps)
 \leq M e^{n\gamma}Z_n(S,\mathcal G,\delta).
\end{equation}
The reverse inequality follows from monotonicity in the separation scale.
Taking normalized upper and lower limits in
\eqref{eq:fixed-scale-comparison}, then letting $\gamma\to0^+$ and
$\eps\to0^+$, shows that both limits at scale $\delta$ equal $P$.
This proves \eqref{eq:abstract-fixed-scale}.  It is the one-sided positively
expansive version of the comparison in
\cite[Section~2.4]{buzzi-kloeckner-leplaideur-nonlinear}.
\end{proof}

\begin{theorem}[Nonlinear variational principle]\label{thm:vp}
Let $T$ be a forward expansive correspondence on a compact metric space and let
$\cE\in C(\Prob(\GammaT),\R)$.  If $(T,\cE)$ has an abundance of ergodic path
measures, then
\begin{equation}\label{eq:vp}
 \underline\Pi_{\mathrm{top}}^{\cE}(T)
 =\overline\Pi_{\mathrm{top}}^{\cE}(T)
 =\sup_{\substack{Q\in\Ker(X;T)\\\mu\in\Inv(X,Q)}}
   \bigl\{h_\mu(Q)+\cE(\gamma_{\mu,Q})\bigr\}.
\end{equation}
\end{theorem}

\begin{proof}
By \cref{lem:edge-lift}, the two topological quantities on the left of
\eqref{eq:vp} are the nonlinear upper and lower pressures of
$(\OmegaT,\sigma,\widehat\cE)$.  We prove the path-space variational formula
directly.

The lower-bound construction adapts
\cite[Proposition~2.3]{buzzi-kloeckner-leplaideur-nonlinear} to the path
space.  Specifically, let $\nu\in\Inv(\OmegaT,\sigma)$ and $\eta>0$.
Abundance supplies an ergodic $\rho$ such that
\[
 h_\rho(\sigma)+\widehat\cE(\rho)
 >h_\nu(\sigma)+\widehat\cE(\nu)-\eta.
\]
Uniform continuity of $\widehat\cE$ gives $a>0$ such that
$d_{\mathrm{BL}}(\lambda,\rho)<2a$ implies
$\widehat\cE(\lambda)>\widehat\cE(\rho)-\eta$.  Put
\[
 B_n(\boldsymbol x,r)
 :=\{\boldsymbol y\in\OmegaT:
       d_{\omega,n}(\boldsymbol x,\boldsymbol y)<r\}.
\]
By Birkhoff's theorem, $\Delta_n(\boldsymbol x)\to\rho$ for
$\rho$-almost every $\boldsymbol x$.  Egorov's theorem therefore gives a
measurable set $A$ with $\rho(A)>3/4$ and $N_A$ such that
\[
 d_{\mathrm{BL}}(\Delta_n(\boldsymbol x),\rho)<a
 \quad(\boldsymbol x\in A,\ n\geq N_A).
\]
We use the following form of the Brin--Katok entropy formula.  Positive
expansivity provides a finite generator, so $h_\rho(\sigma)<\infty$. If
$\rho$ assigns positive mass to a point, ergodicity makes it supported on a periodic orbit and the
formula below is immediate.  In the non-atomic case it is the Brin--Katok
formula~\cite{brin-katok-local-entropy}.  Thus, for $\rho$-almost every
$\boldsymbol x$,
\begin{align*}
 h_\rho(\sigma)
 &=\lim_{r\to0^+}\liminf_{n\to\infty}
   -\frac1n\log\rho(B_n(\boldsymbol x,r))\\
 &=\lim_{r\to0^+}\limsup_{n\to\infty}
   -\frac1n\log\rho(B_n(\boldsymbol x,r)).
\end{align*}
Put
\[
 \ell_r(\boldsymbol x)
 :=\liminf_{n\to\infty}
   -\frac1n\log\rho(B_n(\boldsymbol x,r)).
\]
As $r>0$ decreases to zero, the functions $\ell_r$ increase to
$h_\rho(\sigma)$ almost everywhere.  We may therefore decrease the value of
$a$ chosen above so that
\[
 C:=\{\boldsymbol x:\ell_{2a}(\boldsymbol x)
                 >h_\rho(\sigma)-\eta/2\}
 \quad\text{satisfies}\quad \rho(C)>3/4.
\]
After decreasing $a$, enlarge $N_A$, if necessary, so that the
preceding empirical-measure estimate remains valid on $A$;
this is possible by the uniform convergence provided by
Egorov's theorem.
For $N\geq1$, define
\[
 B_N:=\left\{\boldsymbol x:
   \rho(B_n(\boldsymbol x,2a))
   \leq e^{-n(h_\rho(\sigma)-\eta)}
   \text{ for every }n\geq N\right\}.
\]
The definition of the lower limit gives
$C\subseteq\bigcup_{N\geq1}B_N$, and the sets $B_N$ increase with $N$.
Consequently, some $N_B$ satisfies $\rho(B_{N_B})>3/4$.  With
$B:=B_{N_B}$ we have $\rho(B_n(\boldsymbol x,2a))
 \leq e^{-n(h_\rho(\sigma)-\eta)}
 \quad(\boldsymbol x\in B,\ n\geq N_B).$

For $0<\eps<a$, take a maximal $(n,\eps)$-separated set $E_n$ and a
minimal subset $E_n'\subset E_n$ whose $a$-Bowen balls cover $A\cap B$.
Such a subcover exists because maximality of $E_n$ makes its closed
$\eps$-Bowen balls, and hence its open $a$-Bowen balls, cover $\OmegaT$.
Inclusion-minimality
ensures that every selected ball meets $A\cap B$; choose
$\boldsymbol y_{\boldsymbol x}\in
B_n(\boldsymbol x,a)\cap A\cap B$ for each $\boldsymbol x\in E_n'$.
Then $B_n(\boldsymbol x,a)
 \subset B_n(\boldsymbol y_{\boldsymbol x},2a),$ so every selected ball has $\rho$-measure at most
$e^{-n(h_\rho(\sigma)-\eta)}$.  Since
$\rho(A\cap B)>1/2$, the covering inequality gives $|E_n'|\geq\tfrac12 e^{n(h_\rho(\sigma)-\eta)}.$
Moreover, coupling the $j$th orbit points of $\boldsymbol x$ and
$\boldsymbol y_{\boldsymbol x}$ gives $d_{\mathrm{BL}}
 \bigl(\Delta_n(\boldsymbol x),
       \Delta_n(\boldsymbol y_{\boldsymbol x})\bigr)\leq a.$

Since $\boldsymbol y_{\boldsymbol x}\in A$, the triangle inequality yields
$d_{\mathrm{BL}}(\Delta_n(\boldsymbol x),\rho)<2a$, and therefore
$\widehat\cE(\Delta_n(\boldsymbol x))>\widehat\cE(\rho)-\eta$.
For all $n\geq\max\{N_A,N_B\}$, the contribution of $E_n'$ to the partition
sum is consequently at least $\tfrac12\exp\left(n\bigl(
 h_\rho(\sigma)+\widehat\cE(\rho)-2\eta\bigr)\right).$
Thus
\[
 \liminf_{n\to\infty}\frac1n
  \log\widehat Z_n(\widehat\cE,\eps)
 \geq h_\rho(\sigma)+\widehat\cE(\rho)-2\eta.
\]
Combining this with the choice of $\rho$ gives, for all sufficiently small
$\eps$,
\[
 \liminf_{n\to\infty}\frac1n
  \log\widehat Z_n(\widehat\cE,\eps)
 >h_\nu(\sigma)+\widehat\cE(\nu)-3\eta.
\]
Letting $\eps\to0^+$, then taking the supremum over $\nu$, and finally
letting $\eta\to0^+$ proves the lower bound.

For the upper bound, fix a path-space separation scale $\delta$ smaller than
an expansive constant.  Choose an $(n,\delta)$-separated set $E_n$ whose weight is
at least $\widehat Z_n(\widehat\cE,\delta)/2$, and put
\[
 p_{n,\boldsymbol x}
 :=\frac{\exp(n\widehat\cE(\Delta_n(\boldsymbol x)))}
          {\sum_{\boldsymbol y\in E_n}
           \exp(n\widehat\cE(\Delta_n(\boldsymbol y)))}.
\]
With
$\Xi_n:=\sum_{\boldsymbol x\in E_n}
p_{n,\boldsymbol x}\delta_{\Delta_n(\boldsymbol x)}$, the Gibbs entropy
identity is
\[
\frac1n\log\sum_{\boldsymbol x\in E_n}
  e^{n\widehat\cE(\Delta_n(\boldsymbol x))}
 =\frac1nH((p_{n,\boldsymbol x})_{\boldsymbol x\in E_n})+
 \int_{\Prob(\OmegaT)}\widehat\cE(\rho)\,d\Xi_n(\rho).
\]

Along a subsequence realizing the limsup, compactness gives a further
subsequence $\Xi_n\to\Xi$.  By
\cref{lem:empirical-distribution-entropy} and continuity of
$\widehat\cE$, while the factor $2$ in the near-maximal choice of $E_n$
contributes only $(\log2)/n$, we obtain
\begin{align*}
 \limsup_{n\to\infty}\frac1n
  \log\widehat Z_n(\widehat\cE,\delta)
 &\leq\int_{\Inv(\OmegaT,\sigma)}
   \bigl(h_\rho(\sigma)+\widehat\cE(\rho)\bigr)\,d\Xi(\rho)\\
 &\leq\sup_{\rho\in\Inv(\OmegaT,\sigma)}
   \bigl(h_\rho(\sigma)+\widehat\cE(\rho)\bigr).
\end{align*}
The upper estimate holds at every $\delta$ below an expansive constant, while
the lower estimate holds in the zero-scale limit.  Combining them gives
\[
 \underline\Pi_{\mathrm{top}}^{\cE}(T)
 =\overline\Pi_{\mathrm{top}}^{\cE}(T)
 =\sup_{\nu\in\Inv(\OmegaT,\sigma)}
   \{h_\nu(\sigma)+\widehat\cE(\nu)\}.
\]
The last supremum equals the supremum over stationary transition pairs by
\cref{cor:sup-markov}.
\end{proof}

\begin{lemma}
\label{lem:entropy-usc-positive-expansive}
Let $S:Y\to Y$ be a positively expansive continuous map of a compact metric
space.  Then the entropy map
\[
 \Inv(Y,S)\ni\lambda\longmapsto h_\lambda(S)
\]
is finite-valued, bounded, and upper semicontinuous.
\end{lemma}

\begin{proof}
Let $c>0$ be an expansive constant.  A fixed finite Borel partition
$\mathcal A_0$ with atoms of diameter smaller than $c$ is one-sided
generating, so
$h_\theta(S)\leq\log\#\mathcal A_0$ for every
$\theta\in\Inv(Y,S)$.  This proves a uniform finite bound.

To prove upper semicontinuity, fix $\lambda\in\Inv(Y,S)$ and choose a finite
Borel partition $\mathcal A$ whose
atoms have diameter smaller than $c$ and whose atom boundaries are
$\lambda$-null.  Such a partition is obtained from a finite cover by balls of
radii smaller than $c/2$, choosing the radii so that their boundary spheres
are $\lambda$-null, and then successively taking set differences.

The partition $\mathcal A$ is one-sided generating for every invariant
measure: if two points have the same $\mathcal A$-itinerary, their forward
iterates remain at distance less than $c$, so positive expansivity makes the
points equal.  The Kolmogorov--Sinai theorem therefore gives, for every
$\theta\in\Inv(Y,S)$,
\begin{equation}\label{eq:entropy-finite-generator}
 h_\theta(S)
 =h_\theta(S,\mathcal A)
 =\inf_{m\geq1}\frac1m
   H_\theta\left(\bigvee_{j=0}^{m-1}S^{-j}\mathcal A\right)
 \leq\log\#\mathcal A.
\end{equation}
Now let $\lambda_k\in\Inv(Y,S)$ converge weakly to $\lambda$.  Invariance of $\lambda$ implies that
the boundary of every atom of
$\bigvee_{j=0}^{m-1}S^{-j}\mathcal A$ is $\lambda$-null.  Consequently, for
each fixed $m$ the corresponding finite-partition entropy is continuous along
$\lambda_k\to\lambda$.  From \eqref{eq:entropy-finite-generator},
\[
 \limsup_{k\to\infty}h_{\lambda_k}(S)
 \leq\frac1m
 H_\lambda\left(\bigvee_{j=0}^{m-1}S^{-j}\mathcal A\right)
 \quad(m\geq1).
\]
Taking the infimum over $m$ proves
$\limsup_{k\to\infty}h_{\lambda_k}(S)\leq h_\lambda(S)$.
\end{proof}

\begin{theorem}[Existence of equilibrium pairs]\label{thm:existence}
Let $T$ be a forward expansive correspondence on a compact metric space, let
$\cE\in C(\Prob(\GammaT),\R)$, and assume that $(T,\cE)$ has an abundance of
ergodic path measures.  Then the set of nonlinear equilibrium pairs is
nonempty.  More precisely, applying \cref{prop:markovization} to a maximizing
invariant path measure produces an equilibrium pair.
\end{theorem}

\begin{proof}
The space $\OmegaT$ is nonempty and compact and $\sigma$ is continuous, so the
Krylov--Bogolyubov argument shows that
$\Inv(\OmegaT,\sigma)$ is nonempty; it is compact because it is a closed subset
of $\Prob(\OmegaT)$.  The shift is positively expansive, and hence
\cref{lem:entropy-usc-positive-expansive} shows that
$\nu\mapsto h_\nu(\sigma)$ is upper semicontinuous.  Since
$\widehat\cE$ is continuous, the functional
\[
 \nu\longmapsto h_\nu(\sigma)+\widehat\cE(\nu)
\]
attains its maximum at some $\nu_*$.  By \cref{thm:vp}, its value is
$\Ptop^{\cE}(T)$.

Let $(\mu_*,Q_*)$ be the stationary pair obtained from $\nu_*$ by
\cref{prop:markovization}.  Then
\[
 \cE(\gamma_{\mu_*,Q_*})=\widehat\cE(\nu_*),
 \qquad h_{\mu_*}(Q_*)\geq h_{\nu_*}(\sigma).
\]
The variational identity in \cref{thm:vp} also bounds the resulting
stationary-pair value above by
$\Ptop^{\cE}(T)$.  Hence equality holds throughout, and
$(\mu_*,Q_*)$ is an equilibrium pair.
\end{proof}

\begin{theorem}[Pressure at a fixed scale]\label{thm:fixed-scale}
Let $T$ be a forward expansive correspondence on a compact metric space, let
$\cE\in C(\Prob(\GammaT),\R)$, and assume that $(T,\cE)$ has an abundance of
ergodic path measures.  Let $\eta>0$ be an expansive constant for $T$ and set $\delta_0:=\frac{\eta}{8(1+\eta)}
$ and $\eps_*:=\frac{\delta_0}{4}=\frac{\eta}{32(1+\eta)}.$
Then the following limit exists and
\begin{equation}\label{eq:fixed-scale-pressure}
 \Ptop^{\cE}(T)
 =\lim_{n\to\infty}\frac1n\log Z_n(T,\cE,\eps_*).
\end{equation}

\end{theorem}

\begin{proof}
The calculation following the definition of forward expansivity shows that $c_\omega:=\frac{\eta}{2(1+\eta)}$ is an expansive constant for $(\OmegaT,\sigma)$ \cite[Proposition~6.3]{li-li-zhang-correspondences}.  Since
$\delta_0=c_\omega/4<c_\omega/2$, we may apply
\cref{lem:fixed-scale-stabilization} to the path-space
variational principle in \cref{thm:vp} to obtain
\begin{equation}\label{eq:path-fixed-scale}
 \Ptop^{\cE}(T)
 =\lim_{n\to\infty}\frac1n
   \log\widehat Z_n(\widehat\cE,\delta_0).
\end{equation}
The same argument applies at every smaller positive scale.

Apply \cref{lem:edge-lift}(ii) with $\delta=\delta_0$.
The normalized error
$\omega_{\cE}(2r/(n+r))+rM_{\cE}/n$ tends to zero as $n\to\infty$,
and $a=\delta_0/4=\eps_*$. Thus equations
\eqref{eq:path-fixed-scale} and \eqref{eq:path-to-edge-estimate} imply
\[
 \Ptop^{\cE}(T)
 \leq\liminf_{n\to\infty}\frac1n
      \log Z_n(T,\cE,\eps_*).
\]
On the other hand, \cref{lem:edge-lift}(i) gives
\[
 \limsup_{n\to\infty}\frac1n\log Z_n(T,\cE,\eps_*)
 \leq
 \limsup_{n\to\infty}\frac1n
 \log\widehat Z_n\left(\widehat\cE,
       \frac{\eps_*}{4(1+\eps_*)}\right)
 =\Ptop^{\cE}(T).
\]
The liminf and limsup are equal, proving
\eqref{eq:fixed-scale-pressure}.
\end{proof}

\section{Nonlinear Gibbs ensembles and their projections}\label{sec:gibbs}

For a nonempty $(n,\eps)$-separated set
$\cC\subset\cO_{n+1}(T)$, write
\begin{equation*}
 W_n^{\cE}(\cC)
 :=\sum_{\boldsymbol x\in\cC}
   \exp\bigl(n\cE(L_n(\boldsymbol x))\bigr).
\end{equation*}
The corresponding weighted average of the empirical measures $L_n$ is
the probability measure on $\GammaT$ given by
\begin{equation*}
 \Gamma_{\cC}
 :=\frac{1}{W_n^{\cE}(\cC)}
   \sum_{\boldsymbol x\in\cC}
   \exp\bigl(n\cE(L_n(\boldsymbol x))\bigr)L_n(\boldsymbol x).
\end{equation*}
Its first marginal is the probability measure on $X$ given by
\begin{equation*}
 m_{\cC}
 :=\frac{1}{W_n^{\cE}(\cC)}
   \sum_{\boldsymbol x\in\cC}
   \exp\bigl(n\cE(L_n(\boldsymbol x))\bigr)
   \left(\frac1n\sum_{j=0}^{n-1}\delta_{x_j}\right).
\end{equation*}
These measures are the two-coordinate and first-coordinate projections,
respectively, of a nonlinear Gibbs ensemble on $(\OmegaT,\sigma)$.
The ensemble is constructed by the weighted empirical-measure formula in
\cite[Equation~(1.5)]{buzzi-kloeckner-leplaideur-nonlinear}, after extending
each finite orbit to an infinite orbit; here the normalizing factor is the
actual weight $W_n^{\cE}(\cC)$, whether or not $\cC$ maximizes the partition
sum.  The construction and the projection identity are written out in
\eqref{eq:path-gibbs-ensemble}--\eqref{eq:gibbs-projection-identity} below.
Neither $\Gamma_{\cC}$ nor $m_{\cC}$ depends on these extensions.
The two marginals of $\Gamma_{\cC}$ need not agree: their difference consists
of the initial and final endpoint contributions divided by $n$.  Every
accumulation point as $n\to\infty$ nevertheless has equal marginals, since
\begin{equation*}
 \left\|(\proj_1)_*\Gamma_{\cC}-(\proj_2)_*\Gamma_{\cC}\right\|_{\mathrm{TV}}
 \leq\frac{2}{n}.
\end{equation*}

To describe the limits, introduce the following sets.  We denote by $\mathsf{EM}_{\mathrm{path}}(T,\cE)$ the set of
path-space equilibrium measures; that is, the invariant probability
measures on $\OmegaT$ that maximize the sum of entropy and lifted energy:
\begin{equation*}
\begin{split}
\mathsf{EM}_{\mathrm{path}}(T,\cE)
:=\Bigl\{\nu\in\Inv(\OmegaT,\sigma):\;&
h_\nu(\sigma)+\widehat\cE(\nu)\\
&=\sup_{\rho\in\Inv(\OmegaT,\sigma)}
\bigl(h_\rho(\sigma)+\widehat\cE(\rho)\bigr)
\Bigr\}.
\end{split}
\end{equation*}

We also denote by $\mathsf{J}_{\mathrm{eq}}(T,\cE)$ the set of
two-coordinate distributions induced by nonlinear equilibrium pairs:
\begin{equation*}
\mathsf{J}_{\mathrm{eq}}(T,\cE)
:=
\left\{
\gamma_{\mu,Q}:
(\mu,Q)\text{ is a nonlinear equilibrium pair}
\right\}.
\end{equation*}
\begin{proposition}\label{prop:equilibrium-edge-laws}
Let $T$ be a forward expansive correspondence on a compact metric space, let
$\cE\in C(\Prob(\GammaT),\R)$, and assume that $(T,\cE)$ has an abundance of
ergodic path measures.  Then
\begin{equation}\label{eq:path-edge-equivalence}
 \mathsf{J}_{\mathrm{eq}}(T,\cE)
 =\{(p_2)_*\nu:\nu\in
       \mathsf{EM}_{\mathrm{path}}(T,\cE)\}.
\end{equation}
In particular, $\mathsf{J}_{\mathrm{eq}}(T,\cE)$ is a nonempty compact subset of
$\Prob(\GammaT)$.
\end{proposition}

\begin{proof}
If $\nu$ is a maximizing path measure, the stationary pair $(\mu,Q)$
obtained from its two-coordinate marginal has no smaller entropy by
\cref{prop:markovization}.  Since the variational values agree by
\cref{thm:vp}, the entropy cannot increase strictly; $(\mu,Q)$ is an
equilibrium pair and $(p_2)_*\nu=\gamma_{\mu,Q}$.

Conversely, an equilibrium pair $(\mu,Q)$ induces the Markov path measure
$\nu_{\mu,Q}$.  By \eqref{eq:markov-converse}, its path entropy and energy
are equal to the entropy and energy of the pair, so it is a maximizing path
measure.  This proves \eqref{eq:path-edge-equivalence}.  The path equilibrium
set is compact because entropy is upper semicontinuous and the lifted energy is
continuous.  Compactness of $\mathsf{J}_{\mathrm{eq}}(T,\cE)$ follows from continuity of
$(p_2)_*$.
\end{proof}

\begin{theorem}[Limits of projected Gibbs ensembles]\label{thm:gibbs-ensembles}
Let $T$ be a forward expansive correspondence on a compact metric space, let
$\cE\in C(\Prob(\GammaT),\R)$, and assume that $(T,\cE)$ has an abundance of
ergodic path measures.  Fix the scale $\eps_*$ from
\cref{thm:fixed-scale}.  Let $n_k\to\infty$, and for each $k$ let
$\cC_k\subset\cO_{n_k+1}(T)$ be $(n_k,\eps_*)$-separated and satisfy
\begin{equation}\label{eq:asymptotically-adapted}
 \lim_{k\to\infty}\frac1{n_k}
       \log W_{n_k}^{\cE}(\cC_k)
 =\Ptop^{\cE}(T).
\end{equation}
Then every accumulation point $\Gamma_\infty$ of
$(\Gamma_{\cC_k})_k$ is an average of the two-coordinate distributions of
equilibrium pairs: there exists
$\tau\in\Prob(\mathsf{J}_{\mathrm{eq}}(T,\cE))$ such that
\begin{equation}\label{eq:edge-barycenter}
 \Gamma_\infty=\int_{\mathsf{J}_{\mathrm{eq}}(T,\cE)}\gamma\,d\tau(\gamma).
\end{equation}
Consequently, every accumulation point $m_\infty$ of the first marginals
$(m_{\cC_k})_k$ has the form
\begin{equation}\label{eq:vertex-barycenter}
 m_\infty
 =\int_{\mathsf{J}_{\mathrm{eq}}(T,\cE)}(\proj_1)_*\gamma\,d\tau(\gamma)
\end{equation}
for a probability measure $\tau$ supported on two-coordinate distributions of equilibrium pairs.
\end{theorem}

\begin{proof}
For each $\boldsymbol x\in\cC_k$, choose an arbitrary infinite extension
$e_k(\boldsymbol x)\in\OmegaT$ and set
$\widehat\cC_k:=e_k(\cC_k)$.  Distinct finite orbits have distinct extensions,
so $e_k$ is injective.  Define the path Gibbs ensemble
\begin{equation}\label{eq:path-gibbs-ensemble}
 \nu_{\cC_k}
 :=\frac{1}{W_{n_k}^{\cE}(\cC_k)}
   \sum_{\boldsymbol x\in\cC_k}
   e^{n_k\cE(L_{n_k}(\boldsymbol x))}
   \Delta_{n_k}(e_k(\boldsymbol x)).
\end{equation}
The identity \eqref{eq:empirical-compatibility} gives
\begin{equation}\label{eq:gibbs-projection-identity}
 (p_2)_*\nu_{\cC_k}=\Gamma_{\cC_k}.
\end{equation}
It also gives equality of the finite-orbit and path-space weights:
\begin{equation}\label{eq:gibbs-weight-compatibility}
 W_{n_k}^{\cE}(\cC_k)
 =\sum_{\boldsymbol z\in\widehat\cC_k}
   \exp\bigl(n_k\widehat\cE(\Delta_{n_k}(\boldsymbol z))\bigr).
\end{equation}

By \cref{lem:edge-lift}(i), $\widehat\cC_k$ is an
$(n_k,\delta_*)$-separated subset of $\OmegaT$, where $\delta_*:=\frac{\eps_*}{4(1+\eps_*)}.$ With the notation used in \cref{thm:fixed-scale},
$c_\omega:=\eta/(2(1+\eta))$ is a path expansivity constant and
$\eps_*=c_\omega/16$.  Hence
$\delta_*<\eps_*<c_\omega/2$.  The variational principle
\cref{thm:vp} and
\cref{lem:fixed-scale-stabilization} therefore give the full fixed-scale limit
\[
 \lim_{n\to\infty}\frac1n
 \log\widehat Z_n(\widehat\cE,\delta_*)
 =\Ptop^{\cE}(T).
\]
Together with \eqref{eq:gibbs-weight-compatibility}, this shows that
\eqref{eq:asymptotically-adapted} is precisely the asymptotic optimality
condition for the separated path sets $\widehat\cC_k$.  Write
\[
 p_{k,\boldsymbol x}
 :=\frac{e^{n_k\cE(L_{n_k}(\boldsymbol x))}}
         {W_{n_k}^{\cE}(\cC_k)},
 \qquad
 \Xi_k:=\sum_{\boldsymbol x\in\cC_k}
 p_{k,\boldsymbol x}
 \delta_{\Delta_{n_k}(e_k(\boldsymbol x))}.
\]
Then $\nu_{\cC_k}$ is the barycenter of $\Xi_k$.  Pass to a subsequence along
which $\Gamma_{\cC_k}\to\Gamma_\infty$ and, by compactness, to a further
subsequence such that $\Xi_k\to\Xi$ in
$\Prob(\Prob(\OmegaT))$.  The Gibbs entropy identity gives
\[
 \frac1{n_k}\log W_{n_k}^{\cE}(\cC_k)
 =\frac1{n_k}H((p_{k,\boldsymbol x})_{\boldsymbol x\in\cC_k})
  +\int\widehat\cE(\rho)\,d\Xi_k(\rho).
\]
The limit measure $\Xi$ is supported on $\Inv(\OmegaT,\sigma)$ by
\cref{lem:empirical-distribution-entropy}.  That lemma, continuity of
$\widehat\cE$, the preceding identity, and
\eqref{eq:asymptotically-adapted} give
\begin{align*}
 \Ptop^{\cE}(T)
 &=\lim_{k\to\infty}\frac1{n_k}
      \log W_{n_k}^{\cE}(\cC_k)\\
 &\leq\int_{\Inv(\OmegaT,\sigma)}
   \bigl(h_\rho(\sigma)+\widehat\cE(\rho)\bigr)\,d\Xi(\rho).
\end{align*}
On the other hand, the variational identity in \cref{thm:vp} gives the
pointwise bound
\[
 h_\rho(\sigma)+\widehat\cE(\rho)
 \leq\Ptop^{\cE}(T)
 \qquad(\rho\in\Inv(\OmegaT,\sigma)).
\]
The integral is consequently equal to $\Ptop^{\cE}(T)$.  The nonnegative
Borel function
\[
 G(\rho):=\Ptop^{\cE}(T)-h_\rho(\sigma)-\widehat\cE(\rho)
\]
therefore has zero $\Xi$-integral.  Hence $G=0$ for $\Xi$-almost every
$\rho$.  Since $h_\rho(\sigma)+\widehat\cE(\rho)$ is upper semicontinuous by
\cref{lem:entropy-usc-positive-expansive}, its maximizing set
$\mathsf{EM}_{\mathrm{path}}(T,\cE)=\{G=0\}$ is closed.  Thus $\Xi$ is
concentrated on, and its topological support is contained in,
$\mathsf{EM}_{\mathrm{path}}(T,\cE)$.

The barycenter map on $\Prob(\Prob(\OmegaT))$ is continuous: for
$f\in C(\OmegaT)$, evaluation against $f$ is the continuous function
$\rho\mapsto\int f\,d\rho$.  It follows that $\nu_\infty
 =\int_{\mathsf{EM}_{\mathrm{path}}(T,\cE)}\rho\,d\Xi(\rho)$ is the limit of the path Gibbs ensembles $\nu_{\cC_k}$.  By
\cref{prop:equilibrium-edge-laws}, the image of $\Xi$ under
$\rho\mapsto(p_2)_*\rho$ is a probability measure $\tau$ on
$\mathsf{J}_{\mathrm{eq}}(T,\cE)$, and \eqref{eq:gibbs-projection-identity} together with
$\Gamma_{\cC_k}\to\Gamma_\infty$ gives \eqref{eq:edge-barycenter}.

Finally, let $m_\infty$ be any accumulation point of $(m_{\cC_k})_k$ and
pass to a subsequence on which $m_{\cC_k}\to m_\infty$.  Compactness of
$\Prob(\GammaT)$ provides a further subsequence with
$\Gamma_{\cC_k}\to\Gamma_\infty$.  Since
$m_{\cC_k}=(\proj_1)_*\Gamma_{\cC_k}$, continuity of projection gives
$m_\infty=(\proj_1)_*\Gamma_\infty$.  Applying $\proj_1$ to
\eqref{eq:edge-barycenter} proves \eqref{eq:vertex-barycenter}.
\end{proof}

\begin{corollary}\label{cor:gibbs-unique}
Assume the hypotheses and notation of \cref{thm:gibbs-ensembles}.  If
$\mathsf{J}_{\mathrm{eq}}(T,\cE)=\{\gamma_*\}$, then every sequence satisfying
\eqref{eq:asymptotically-adapted} obeys
\begin{equation*}
 \Gamma_{\cC_k}\xrightarrow{\weakstar}\gamma_*,
 \qquad
 m_{\cC_k}\xrightarrow{\weakstar}(\proj_1)_*\gamma_*.
\end{equation*}
If there are finitely many two-coordinate distributions of equilibrium pairs, every accumulation point of $(\Gamma_{\cC_k})_k$ lies in their convex hull.
\end{corollary}

\begin{proof}
In the unique case, \cref{thm:gibbs-ensembles} forces every accumulation point
to equal $\gamma_*$.  Compactness of $\Prob(\GammaT)$ then gives convergence of
the whole sequence; convergence of the first marginals follows by projection.  The final
statement is immediate from \eqref{eq:edge-barycenter}.
\end{proof}

\begin{remark}
Even though every $\Gamma_\infty$ in \eqref{eq:edge-barycenter} has equal
marginals and therefore disintegrates into a stationary transition pair, that
pair need not be an equilibrium pair.  Indeed, the nonlinear energy of a
barycenter need not equal the barycenter of the energies.  This is the
correspondence analogue of the multiple-equilibrium phenomenon in the
Curie--Weiss model.
\end{remark}

\section{Energies with finitely many two-coordinate potentials}\label{sec:legendre}

\subsection{Rotation sets and localized entropy}

Let
$\boldsymbol\psi=(\psi_1,\ldots,\psi_d)\in C(\GammaT,\R^d)$ and let
$F:U\to\R$ be continuous on an open neighborhood of the convex hull of
$\boldsymbol\psi(\GammaT)$.  This finite-dimensional viewpoint is also closely
related to the rotation-set interpretation of nonlinear thermodynamic formalism
developed in~\cite{kucherenko-rotation-nonlinear}.  Consider
\begin{equation*}
 \cE_{F,\boldsymbol\psi}(\gamma)
 :=F\left(\int_{\GammaT}\boldsymbol\psi\,d\gamma\right).
\end{equation*}
We abbreviate
$\Ptop^{F,\boldsymbol\psi}(T):=\Ptop^{\cE_{F,\boldsymbol\psi}}(T)$.
Define the rotation set
\begin{equation*}
 \mathcal R(\boldsymbol\psi)
 :=\left\{
   \int\boldsymbol\psi\,d\gamma_{\mu,Q}:
   Q\in\Ker(X;T),\ \mu\in\Inv(X,Q)
 \right\}
\end{equation*}
and, for $\boldsymbol z\in\mathcal R(\boldsymbol\psi)$, the localized kernel entropy
\begin{equation*}
 \mathfrak h(\boldsymbol z)
 :=\sup\left\{h_\mu(Q):
   \int_{\GammaT}\boldsymbol\psi\,d\gamma_{\mu,Q}=\boldsymbol z
 \right\},
\end{equation*}
where the supremum ranges over stationary transition pairs.
The vector $\int_{\GammaT}\boldsymbol\psi\,d\gamma_{\mu,Q}$ is the
\emph{rotation vector} of the pair. Rotation vectors of nonlinear equilibrium
pairs are called \emph{equilibrium values}.

\begin{proposition}[Finite-dimensional reduction]
\label{prop:finite-reduction}
Let $T$ be a forward expansive correspondence on a compact metric space, let
$\boldsymbol\psi\in C(\GammaT,\R^d)$, and let $F$ be continuous on an open
neighborhood of the convex hull of $\boldsymbol\psi(\GammaT)$.  Assume that
$(T,\cE_{F,\boldsymbol\psi})$ has an abundance of ergodic path measures.  Then
\begin{equation}\label{eq:finite-reduction}
 \Ptop^{F,\boldsymbol\psi}(T)
 =\max_{\boldsymbol z\in\mathcal R(\boldsymbol\psi)}
  \bigl\{\mathfrak h(\boldsymbol z)+F(\boldsymbol z)\bigr\}.
\end{equation}
\end{proposition}

\begin{proof}
Let $\mathcal S_T$ be the set of probability measures on $\GammaT$ with equal
first and second marginals.  It is compact: $\Prob(\GammaT)$ is compact, and
equality of the marginals is closed because it is equivalent to
\[
 \int_{\GammaT}(f(x)-f(y))\,d\gamma(x,y)=0
 \quad\text{for every }f\in C(X).
\]
Let $\gamma\in\mathcal S_T$ and denote its common marginal by $\mu$.
Disintegration over the first coordinate gives
$\gamma(dx,dy)=\mu(dx)Q_x^0(dy)$, with
$Q_x^0(T(x))=1$ for $\mu$-almost every $x$.  The Borel-selector modification
used explicitly in the proof of \cref{prop:markovization} produces
$Q\in\Ker(X;T)$ without changing the disintegration.  Since the second
marginal is also $\mu$,
\[
 (\mu Q)(A)=\gamma(X\times A)=\mu(A),
\]
so $\mu\in\Inv(X,Q)$ and $\gamma=\gamma_{\mu,Q}$.  Conversely, every
stationary pair has a two-coordinate distribution in $\mathcal S_T$.  Therefore
$\mathcal R(\boldsymbol\psi)$ is exactly the continuous image of
$\mathcal S_T$ under
$\gamma\mapsto\int\boldsymbol\psi\,d\gamma$, and is compact.

For a stationary pair with rotation vector $\boldsymbol z$, the definition of
$\mathfrak h$ gives
\[
 h_\mu(Q)+F(\boldsymbol z)
 \leq\mathfrak h(\boldsymbol z)+F(\boldsymbol z).
\]
Conversely, for every $\eta>0$ the definition of the supremum supplies a
stationary pair at $\boldsymbol z$ whose entropy is greater than
$\mathfrak h(\boldsymbol z)-\eta$.  Hence
\begin{align*}
 &\sup_{\substack{Q\in\Ker(X;T)\\\mu\in\Inv(X,Q)}}
 \left\{h_\mu(Q)+
 F\left(\int\boldsymbol\psi\,d\gamma_{\mu,Q}\right)\right\}\\
 &\qquad=
 \sup_{\boldsymbol z\in\mathcal R(\boldsymbol\psi)}
 \{\mathfrak h(\boldsymbol z)+F(\boldsymbol z)\}.
\end{align*}
The variational principle \cref{thm:vp} identifies the left-hand side with
the nonlinear topological
pressure, so
\[
 \Ptop^{F,\boldsymbol\psi}(T)
 =\sup_{\boldsymbol z\in\mathcal R(\boldsymbol\psi)}
  \{\mathfrak h(\boldsymbol z)+F(\boldsymbol z)\}.
\]
By \cref{thm:existence}, an equilibrium pair exists.  Its rotation vector
attains the last supremum, which is consequently a maximum and proves
\eqref{eq:finite-reduction}.
\end{proof}

\subsection{Legendre duality}

For $\boldsymbol y\in\R^d$, let
\begin{equation}\label{eq:linear-pressure-function}
 P(\boldsymbol y)
 :=P(T,\boldsymbol y\cdot\boldsymbol\psi)
 =\sup_{\boldsymbol z\in\mathcal R(\boldsymbol\psi)}
    \bigl(\mathfrak h(\boldsymbol z)+\boldsymbol y\cdot\boldsymbol z\bigr).
\end{equation}
Indeed, the energy
$\gamma\mapsto\boldsymbol y\cdot\int\boldsymbol\psi\,d\gamma$ is affine and
therefore convex.  It has an abundance of ergodic path measures by
\cref{prop:abundance-criteria}(ii).  Applying the variational principle
\cref{thm:vp} to this energy and
grouping stationary pairs according to their rotation vectors, as in the
preceding proof, gives both equalities in
\eqref{eq:linear-pressure-function}.

For a two-coordinate potential $\varphi\in C(\GammaT,\R)$, let
$\Eq(T,\varphi)$ denote the stationary transition pairs attaining its linear
variational pressure.

\begin{definition}\label{def6.2}
Let $r\in\N\cup\{\infty,\omega\}$, where $r=\omega$ denotes real analyticity.
The \emph{$C^r$ Legendre assumptions} for $(T,\boldsymbol\psi)$ are:
\begin{enumerate}[label=\textup{(\roman*)}]
 \item $\mathcal R(\boldsymbol\psi)$ has nonempty interior;
 \item $\mathfrak h$ is finite and upper semicontinuous on the rotation set,
       and is $C^r$ and strictly concave on its interior;
 \item $P$ is $C^r$ and strictly convex, and
       $\nabla P:\R^d\to\intt\mathcal R(\boldsymbol\psi)$ is a
       homeomorphism; for $r\ne1$, this map and its inverse are $C^{r-1}$;
 \item the Legendre identity
 \[
  \mathfrak h(\boldsymbol z)
  =\inf_{\boldsymbol y\in\R^d}
    \bigl(P(\boldsymbol y)-\boldsymbol y\cdot\boldsymbol z\bigr)
 \]
 holds on the rotation set.
\end{enumerate}
These are the finite-dimensional Legendre hypotheses used below; compare
the classical Legendre--Fenchel theory in \cite{rockafellar-convex-analysis}
and its thermodynamic formulation in
\cite[Definitions~3.8 and 3.10]{buzzi-kloeckner-leplaideur-nonlinear}.
For $r\in\{\infty,\omega\}$, the notation $C^{r-1}$ in item~\textup{(iii)}
means $C^r$.
\end{definition}

We say that $(T,\boldsymbol\psi)$ has \emph{unique linear
equilibrium pairs} if $\Eq(T,\boldsymbol y\cdot\boldsymbol\psi)$ consists of one
equivalence class for every $\boldsymbol y\in\R^d$, where kernels are identified
when they agree almost everywhere with respect to their stationary measure.

\begin{theorem}[Reduction to linear equilibrium pairs]
\label{thm:legendre}
Let $T$ be a forward expansive correspondence on a compact metric space, let
$\boldsymbol\psi\in C(\GammaT,\R^d)$, and assume that
$(T,\boldsymbol\psi)$ satisfies the $C^1$ Legendre assumptions.  Let $U$ be
an open neighborhood of the convex hull of $\boldsymbol\psi(\GammaT)$ and let
$F\in C^1(U)$.  Assume that
$(T,\cE_{F,\boldsymbol\psi})$ has an abundance of ergodic path measures and
suppose that the maximizing set
\[
 V:=\argmax_{\boldsymbol z\in\mathcal R(\boldsymbol\psi)}
   \bigl(\mathfrak h(\boldsymbol z)+F(\boldsymbol z)\bigr)
\]
is contained in $\intt\mathcal R(\boldsymbol\psi)$.  Set
$Y=(\nabla P)^{-1}(V)$.  Then the nonlinear equilibrium pairs are exactly
\begin{equation}\label{eq:union-linear-eq}
 \bigcup_{\boldsymbol y\in Y}
 \Eq(T,\boldsymbol y\cdot\boldsymbol\psi).
\end{equation}
Moreover, if
$\boldsymbol z=\nabla P(\boldsymbol y)\in V$, then
\begin{equation}\label{eq:self-consistency}
 \boldsymbol y=\nabla F(\boldsymbol z),
 \qquad
 \boldsymbol z=\nabla P(\nabla F(\boldsymbol z)).
\end{equation}
\end{theorem}

\begin{proof}
By \cref{prop:finite-reduction}, the set $V$ is nonempty.  It is compact
because $\mathfrak h+F$ is upper semicontinuous on the compact rotation set.
The hypothesis $V\subset\intt\mathcal R(\boldsymbol\psi)$ and continuity of
$(\nabla P)^{-1}$ then show that $Y=(\nabla P)^{-1}(V)$ is also nonempty and
compact.

Fix $\boldsymbol z\in\intt\mathcal R(\boldsymbol\psi)$ and set
$\boldsymbol y=(\nabla P)^{-1}(\boldsymbol z)$.  Consider the strictly convex
function
\[
 \boldsymbol v\longmapsto
 P(\boldsymbol v)-\boldsymbol v\cdot\boldsymbol z.
\]
Its gradient vanishes at $\boldsymbol v=\boldsymbol y$ because
$\nabla P(\boldsymbol y)=\boldsymbol z$.  Thus $\boldsymbol y$ is its unique
minimizer, and the assumed Legendre identity becomes
\begin{equation}\label{eq:attained-legendre-identity}
 \mathfrak h(\boldsymbol z)
 =P(\boldsymbol y)-\boldsymbol y\cdot\boldsymbol z.
\end{equation}

Let $(\mu,Q)$ be any stationary pair with rotation vector
$\boldsymbol z$.  The linear variational formula
\eqref{eq:linear-pressure-function} gives
\[
 h_\mu(Q)+\boldsymbol y\cdot\boldsymbol z
 \leq P(\boldsymbol y),
\]
with equality precisely when $(\mu,Q)$ is a linear equilibrium pair for
$\boldsymbol y\cdot\boldsymbol\psi$.  Subtracting
$\boldsymbol y\cdot\boldsymbol z$ and using
\eqref{eq:attained-legendre-identity} proves the exact equivalence
\begin{equation}\label{eq:localized-linear-equivalence}
 h_\mu(Q)=\mathfrak h(\boldsymbol z)
 \quad\Longleftrightarrow\quad
 (\mu,Q)\in\Eq(T,\boldsymbol y\cdot\boldsymbol\psi).
\end{equation}

Suppose first that $(\mu,Q)$ is a nonlinear equilibrium pair and let
$\boldsymbol z$ be its rotation vector.  Its value is
\[
 \Ptop^{F,\boldsymbol\psi}(T)
 =h_\mu(Q)+F(\boldsymbol z)
 \leq\mathfrak h(\boldsymbol z)+F(\boldsymbol z)
 \leq\max_{\boldsymbol w\in\mathcal R(\boldsymbol\psi)}
       \{\mathfrak h(\boldsymbol w)+F(\boldsymbol w)\}.
\]
The first and last terms are equal by
\cref{prop:finite-reduction}; hence both inequalities are equalities.  Thus
$\boldsymbol z\in V$ and
$h_\mu(Q)=\mathfrak h(\boldsymbol z)$.  Since $V$ lies in the interior,
\eqref{eq:localized-linear-equivalence} applies with
$\boldsymbol y=(\nabla P)^{-1}(\boldsymbol z)\in Y$.  Therefore $(\mu,Q)$
belongs to the right-hand side of \eqref{eq:union-linear-eq}.

Conversely, take $\boldsymbol y\in Y$, put
$\boldsymbol z=\nabla P(\boldsymbol y)\in V$, and let
$(\mu,Q)\in\Eq(T,\boldsymbol y\cdot\boldsymbol\psi)$.
Write
$\boldsymbol z_{\mu,Q}:=\int\boldsymbol\psi\,d\gamma_{\mu,Q}$.  For every
$\boldsymbol v\in\R^d$, the variational formula gives
\begin{align*}
 P(\boldsymbol v)
 &\geq h_\mu(Q)+\boldsymbol v\cdot\boldsymbol z_{\mu,Q}\\
 &=P(\boldsymbol y)
   +(\boldsymbol v-\boldsymbol y)\cdot\boldsymbol z_{\mu,Q}.
\end{align*}
Thus $\boldsymbol z_{\mu,Q}$ is a subgradient of $P$ at
$\boldsymbol y$.  Since $P$ is differentiable there,
$\boldsymbol z_{\mu,Q}=\nabla P(\boldsymbol y)=\boldsymbol z$.  Equation
\eqref{eq:localized-linear-equivalence} now gives
$h_\mu(Q)=\mathfrak h(\boldsymbol z)$.  Since $\boldsymbol z$ maximizes
$\mathfrak h+F$, this pair attains the nonlinear pressure and is a nonlinear
equilibrium pair.  This proves both inclusions in
\eqref{eq:union-linear-eq}.

To derive the self-consistency equation, put
$\boldsymbol y(\boldsymbol z):=(\nabla P)^{-1}(\boldsymbol z)$.  On the
interior, \eqref{eq:attained-legendre-identity} reads
\[
 \mathfrak h(\boldsymbol z)
 =P(\boldsymbol y(\boldsymbol z))
  -\boldsymbol y(\boldsymbol z)\cdot\boldsymbol z.
\]
Fix a direction $\boldsymbol u$ and write
$\boldsymbol y_t:=\boldsymbol y(\boldsymbol z+t\boldsymbol u)$.  The
minimizing property in the Legendre formula gives
\begin{align*}
 \mathfrak h(\boldsymbol z+t\boldsymbol u)-\mathfrak h(\boldsymbol z)
 &\leq-t\boldsymbol y(\boldsymbol z)\cdot\boldsymbol u,\\
 \mathfrak h(\boldsymbol z+t\boldsymbol u)-\mathfrak h(\boldsymbol z)
 &\geq-t\boldsymbol y_t\cdot\boldsymbol u.
\end{align*}
Since $(\nabla P)^{-1}$ is continuous, $\boldsymbol y_t\to
\boldsymbol y(\boldsymbol z)$.  Dividing by $t$ (with the inequalities
reversed when $t<0$) and letting $t\to0$ yields
\[
 D\mathfrak h(\boldsymbol z)[\boldsymbol u]
 =-\boldsymbol y(\boldsymbol z)\cdot\boldsymbol u.
\]
Hence $\nabla\mathfrak h(\boldsymbol z)=-\boldsymbol y$.  Because every
$\boldsymbol z\in V$ is an interior maximizer of the differentiable function
$\mathfrak h+F$, we have
$0=\nabla\mathfrak h(\boldsymbol z)+\nabla F(\boldsymbol z)$, and therefore
$\boldsymbol y=\nabla F(\boldsymbol z)$.  Substitution into
$\boldsymbol z=\nabla P(\boldsymbol y)$ proves
\eqref{eq:self-consistency}.
\end{proof}

\begin{corollary}\label{cor:uniqueness}
Under the hypotheses of \cref{thm:legendre}, suppose that
$\mathfrak h+F$ has a unique maximizer $\boldsymbol z_*$ and that the linear
potential $\nabla F(\boldsymbol z_*)\cdot\boldsymbol\psi$ has a unique
equilibrium pair.  Then the nonlinear equilibrium pair is unique.
\end{corollary}

\begin{proof}
By \cref{thm:legendre}, every nonlinear equilibrium pair is a linear
equilibrium pair at the unique maximizing value $\boldsymbol z_*$.  The assumed
linear uniqueness leaves exactly one such pair.
\end{proof}

\subsection{Analytic finiteness}

\begin{theorem}[Analytic finiteness in one dimension]
\label{thm:analytic-finiteness}
Let $T$ be a forward expansive correspondence on a compact metric space and
let $\psi\in C(\GammaT,\R)$.  Assume that $(T,\psi)$ satisfies the $C^\omega$ Legendre assumptions and
has unique linear equilibrium pairs.  Let
$F:U\to\R$ be real analytic on an open neighborhood $U$ of the convex hull of
$\psi(\GammaT)$, and assume that
$(T,\cE_{F,\psi})$ has an abundance of ergodic path measures.  Then the set of
nonlinear equilibrium values is finite.  Consequently, there are only finitely
many nonlinear equilibrium pairs and associated two-coordinate distributions.
\end{theorem}

\begin{proof}
Write $\mathcal R(\psi)=[a,b]$.  Its interior is nonempty by the Legendre
hypothesis.  Let
\[
 g(z):=\mathfrak h(z)+F(z),
 \qquad
 V:=\argmax_{z\in[a,b]}g(z).
\]
By \cref{prop:finite-reduction}, $V$ is the set of nonlinear equilibrium
values.

We first record the endpoint regularity needed below.  A finite concave
function on $[a,b]$ is continuous on $(a,b)$ and has one-sided limits at the
endpoints.  Concavity gives
$\mathfrak h(a)\leq\liminf_{z\downarrow a}\mathfrak h(z)$ and
$\mathfrak h(b)\leq\liminf_{z\uparrow b}\mathfrak h(z)$, while upper
semicontinuity gives the reverse inequalities.  Thus $\mathfrak h$, and hence
$g$, is continuous on all of $[a,b]$.  Consequently $V$ is nonempty and
compact.

The endpoints cannot maximize $g$.  Put $y(z):=(P')^{-1}(z)$.  The attained
Legendre identity proved in \eqref{eq:attained-legendre-identity} reads
$\mathfrak h(z)=P(y(z))-y(z)z$.  Differentiating it and using
$P'(y(z))=z$ cancels the two terms containing $y'(z)$ and gives
\begin{equation*}
 \mathfrak h'(z)=-(P')^{-1}(z).
\end{equation*}
Since $P':\R\to(a,b)$ is an increasing diffeomorphism,
\begin{equation*}
 \lim_{z\downarrow a}\mathfrak h'(z)=+\infty,
 \qquad
 \lim_{z\uparrow b}\mathfrak h'(z)=-\infty.
\end{equation*}
The derivative $F'$ is bounded on a neighborhood of $[a,b]$.  Hence $g'>0$
on $(a,a+\delta)$ and $g'<0$ on $(b-\delta,b)$ for some $\delta>0$.
The mean value theorem shows that $g$ is strictly increasing on the first
interval and strictly decreasing on the second.  Continuity at the endpoints
then gives $g(z)>g(a)$ for $a<z<a+\delta$ and
$g(z)>g(b)$ for $b-\delta<z<b$. Hence neither endpoint maximizes $g$,
so $V$ is a compact subset of $(a,b)$.

Because $P'$ is a diffeomorphism, its derivative does not vanish; because
$P$ is real analytic, the analytic inverse function theorem applies at every
point.  The local analytic inverses agree with the global inverse, so
$(P')^{-1}$ is real analytic on $(a,b)$.  Therefore
\begin{equation*}
 H(z):=g'(z)
 =F'(z)-(P')^{-1}(z)
\end{equation*}
is real analytic on $(a,b)$, and every point of $V$ is a zero of $H$.  The
function $H$ is not identically zero.  Indeed, if it vanished identically, then
$g$ would be constant on $(a,b)$.  The endpoint continuity proved above then
makes $g$ constant on all of $[a,b]$.  Both endpoints would lie in $V$,
contradicting the exclusion of the endpoints.

A nonzero real-analytic function has isolated zeros.  An infinite subset of a
compact interval has an accumulation point; hence the compact set $V$ cannot
be an infinite subset of the isolated zero set of $H$.  Thus $V$ is finite.
For each $z\in V$,
\cref{thm:legendre} identifies all nonlinear equilibrium pairs with the linear
equilibrium pairs for
\[
 y\psi,\qquad y=(P')^{-1}(z)=F'(z).
\]
There is exactly one such pair by hypothesis.  Hence the set of nonlinear
equilibrium pairs is finite, and its image under
$(\mu,Q)\mapsto\gamma_{\mu,Q}$---the set of two-coordinate distributions of equilibrium pairs---is finite
as well.
\end{proof}

For a forward expansive correspondence, a continuous potential $\psi$ is
\emph{Bowen summable} if there are an expansive constant $\eps>0$ and
$K<\infty$ such that
\[
 \left|\sum_{j=0}^{n-1}
  \bigl(\psi(x_j,x_{j+1})-\psi(x'_j,x'_{j+1})\bigr)\right|\leq K
\]
whenever $n\geq1$, $\boldsymbol x,\boldsymbol x'\in\cO_{n+1}(T)$, and
$d_{n+1}(\boldsymbol x,\boldsymbol x')<\eps$; see
\cite[Definition~7.4]{li-li-zhang-correspondences}.

\begin{corollary}\label{cor:analytic-ruelle}
Let $T$ be a forward expansive correspondence with specification on a compact
metric space, and let $\psi\in C(\GammaT,\R)$ be Bowen summable.
Suppose that $(T,\psi)$ satisfies the $C^\omega$ Legendre assumptions.
Let $F$ be real analytic on an open neighborhood of the convex hull of
$\psi(\GammaT)$. Then the energy
$\cE(\gamma)=F(\int_{\GammaT}\psi\,d\gamma)$ has only finitely many nonlinear
equilibrium pairs.
\end{corollary}
\begin{proof}
The energy is continuous, and \cref{cor:specification-abundance} gives
abundance. For each $y\in\R$, the displayed bound for $\psi$ gives the
bound $|y|K$ for $y\psi$, so every $y\psi$ is Bowen summable.
By \cite[Theorem~B]{li-li-zhang-correspondences}, it has a unique linear
equilibrium pair, with kernels identified almost everywhere with respect
to the stationary measure. Thus \cref{thm:analytic-finiteness} applies.
\end{proof}

\begin{remark}
Real analyticity of $F$ cannot in general be replaced by
$C^\infty$ regularity, even with $T$ and $\psi$ as in
\cref{cor:analytic-ruelle}. Write $\mathcal R(\psi)=[a,b]$ and fix a nonempty
compact set $K\subset(a,b)$.

Choose a nonnegative function $q\in C^\infty(\mathbb R)$ with
$q^{-1}(0)=K$, and a smooth function
$\chi:\mathbb R\to[0,1]$ with compact support in $(a,b)$ and
$\chi=1$ on a neighborhood of $K$.
Take $M>\max_{z\in[a,b]}\mathfrak h(z)$ and define
\[
F(z):=
\begin{cases}
-\chi(z)\bigl(\mathfrak h(z)+q(z)\bigr)-(1-\chi(z))M,
   & z\in(a,b),\\
-M, & z\notin(a,b).
\end{cases}
\]
Since $\mathfrak h$ is smooth on $(a,b)$ and $\chi$ vanishes near the
endpoints, $F\in C^\infty(\mathbb R)$.
For every $z\in[a,b]$,
\[
\mathfrak h(z)+F(z)
   =-\chi(z)q(z)+(1-\chi(z))(\mathfrak h(z)-M)
   \le0,
\]
with equality exactly when $z\in K$. Hence
\[
\operatorname*{argmax}_{z\in[a,b]}\bigl(\mathfrak h(z)+F(z)\bigr)=K.
\]
This is the cutoff construction of
\cite[Proposition~3.23]{buzzi-kloeckner-leplaideur-nonlinear}.
\cref{cor:specification-abundance} supplies abundance, and \cref{thm:legendre}, together
with linear uniqueness, gives exactly one nonlinear equilibrium
pair for each $z\in K$. Thus finite, countably infinite, and
uncountable equilibrium sets can occur.

For $d\ge2$, analytic critical sets need not be discrete.
The isolated-zero argument used in \cref{thm:analytic-finiteness} therefore does
not establish finiteness in higher dimensions.
\end{remark}

\subsection{Continuous transitions and Taylor expansions}

We will use a differentiation rule for maxima. Let $K$ be a nonempty compact
metric space, let $a:K\to\R$ be upper semicontinuous, and let $b:K\to\R$
be continuous. For $t\in\R$, set
\[
 v(t):=\max_{z\in K}\{a(z)+tb(z)\},
 \qquad
 M(t):=\operatorname*{argmax}_{z\in K}\{a(z)+tb(z)\}.
\]
The left and right derivatives satisfy
\begin{equation}\label{eq:maxima-one-sided-derivatives}
 v'_-(t)=\min_{z\in M(t)}b(z),
 \qquad
 v'_+(t)=\max_{z\in M(t)}b(z).
\end{equation}
Indeed, for $s>0$, $z\in M(t)$ and $z_s\in M(t+s)$, maximality gives
\[
 b(z)\leq\frac{v(t+s)-v(t)}s\leq b(z_s).
\]
The function $v$ is Lipschitz because $b$ is bounded. Compactness and upper
semicontinuity imply that every limit point of $z_s$ as $s\to0^+$ belongs
to $M(t)$, which proves the right-derivative formula. Applying the same
inequality at $t-s$ and $t$ gives the left-derivative formula.

\begin{theorem}[A Taylor criterion for continuous transitions]
\label{thm:taylor-critical-exponent}
Let $T$ be a forward expansive correspondence on a compact metric space, let
$\psi\in C(\GammaT,\R)$, and assume that $(T,\psi)$ satisfies the $C^\omega$ Legendre assumptions
and has unique linear equilibrium pairs.  Suppose that the rotation interval is
symmetric about $0$, that the linear
pressure $P$ is even, and that the response ratio
\begin{equation*}
 \mathscr R(y):=\frac{y}{P'(y)}
\end{equation*}
is strictly increasing for $y>0$.  Assume that, for some $a,b>0$ and
$k\in\N$ with $k\geq1$,
\begin{equation}\label{eq:response-taylor-jet}
 P'(y)=ay-by^{2k+1}+O(y^{2k+3})
 \qquad(y\to0).
\end{equation}
For $\beta\geq0$, set
\begin{equation}\label{eq:quadratic-family}
 \cE_\beta(\gamma)
 :=\frac\beta2\left(\int_{\GammaT}\psi\,d\gamma\right)^2
\end{equation}
and write $\Phi(\beta):=\Ptop^{\cE_\beta}(T)$.  Then the critical inverse
temperature is
\begin{equation*}
 \beta_c=\frac1a.
\end{equation*}
More precisely:
\begin{enumerate}[label=\textup{(\roman*)}]
\item If $0\leq\beta\leq\beta_c$, the unique nonlinear equilibrium
pair is the linear equilibrium pair at $y=0$, and its rotation value is zero.
\item If $\beta>\beta_c$, there are exactly two nonlinear
equilibrium pairs: the linear equilibrium pairs at $y=\pm y_\beta$, where
$y_\beta>0$ is the unique solution of
\begin{equation*}
 y_\beta=\beta P'(y_\beta).
\end{equation*}
Their rotation values are $\pm z_\beta$, with $z_\beta:=P'(y_\beta)$.
\item As $\beta\downarrow\beta_c$,
\begin{align}
 y_\beta
 &\sim\left(\frac{a^2}{b}\right)^{1/(2k)}
       (\beta-\beta_c)^{1/(2k)},
 \label{eq:general-dual-exponent}\\
 z_\beta
 &\sim C\,(\beta-\beta_c)^{1/(2k)},
 \qquad
 C:=a\left(\frac{a^2}{b}\right)^{1/(2k)}.
 \label{eq:general-order-exponent}
\end{align}
\item The pressure satisfies
\begin{equation}\label{eq:general-nonlinear-pressure}
 \Phi(\beta)=
 \begin{cases}
  P(0),&0\leq\beta\leq\beta_c,\\
  P(y_\beta)-\dfrac\beta2z_\beta^2,&\beta>\beta_c,
 \end{cases}
\end{equation}
and, as $\beta\downarrow\beta_c$, with $A:=C^2/2$,
\begin{align}
 \Phi'(\beta)
 &\sim A(\beta-\beta_c)^{1/k},
 \label{eq:general-pressure-derivative-singularity}\\
 \Phi(\beta)-P(0)
 &\sim\frac{k}{k+1}A
       (\beta-\beta_c)^{1+1/k}.
 \label{eq:general-pressure-singularity}
\end{align}
In particular, $\Phi$ is $C^1$ but not $C^2$ at
$\beta=\beta_c$.
\end{enumerate}
\end{theorem}

\begin{proof}
The energy \eqref{eq:quadratic-family} is convex, so abundance follows from
\cref{prop:abundance-criteria}.  Let the rotation interval be $[-r,r]$ and let
$y(z)=(P')^{-1}(z)$.  Since $P$ is even, $P'$ and $y(z)$ are odd.  Legendre
duality gives $\mathfrak h'(z)=-y(z)$, and therefore
\begin{equation}\label{eq:general-variational-derivative}
 \frac{d}{dz}\left(\mathfrak h(z)+\frac\beta2z^2\right)
 =\beta z-y(z)
 =z\left(\beta-\frac{y(z)}z\right),\qquad z\ne0.
\end{equation}
By hypothesis, $y(z)/z=\mathscr R(y(z))$ is strictly increasing on $(0,r)$.
Its lower endpoint limit is
\begin{equation*}
 \lim_{z\to0^+}\frac{y(z)}z=\frac1{P''(0)}=\frac1a,
\end{equation*}
and its upper endpoint limit is $+\infty$, since the rotation interval is
compact while $y(z)\to+\infty$ as $z\uparrow r$.

The sign pattern in \eqref{eq:general-variational-derivative} shows that zero
is the unique global maximizer for $\beta\leq1/a$.  For $\beta>1/a$, there is
a unique positive critical value $z_\beta$; the variational function increases
on $(0,z_\beta)$ and decreases on $(z_\beta,r)$, so its only global maximizers
are $\pm z_\beta$.  The critical equation is
$y_\beta=\beta z_\beta=\beta P'(y_\beta)$.  The reduction theorem
\cref{thm:legendre} and uniqueness of the linear equilibrium pairs prove
parts~\textup{(i)}--\textup{(ii)}.

Put $\delta=\beta-1/a$.  Dividing
\eqref{eq:response-taylor-jet} by $y$ in the self-consistency equation gives
\[
 1=\beta\left(a-by_\beta^{2k}+O(y_\beta^{2k+2})\right).
\]
Consequently,
\begin{equation*}
 y_\beta^{2k}
 \sim\frac{a^2}{b}\,\delta.
\end{equation*}
This proves \eqref{eq:general-dual-exponent}; multiplying by the leading
response coefficient $a$ proves \eqref{eq:general-order-exponent}.

At a nonzero maximizer, Legendre duality and $y_\beta=\beta z_\beta$ give
\[
 \mathfrak h(z_\beta)+\frac\beta2z_\beta^2
 =P(y_\beta)-\frac\beta2z_\beta^2,
\]
which proves \eqref{eq:general-nonlinear-pressure}. Both maximizing values
$\pm z_\beta$ have the same square. Applying
\eqref{eq:maxima-one-sided-derivatives} with $K=[-r,r]$,
$a(z)=\mathfrak h(z)$ and $b(z)=z^2/2$ therefore gives
\begin{equation*}
 \Phi'(\beta)=\frac12z_\beta^2
 \qquad(\beta>\beta_c),
\end{equation*}
whereas $\Phi'=0$ below the transition.  Equation
\eqref{eq:general-order-exponent} proves
\eqref{eq:general-pressure-derivative-singularity}; integration proves
\eqref{eq:general-pressure-singularity}.  Thus $\Phi'$ is continuous.  If
$k=1$, its one-sided derivatives at $\beta_c$ are $0$ and $A$; if
$k>1$, its right derivative diverges.  In either case $\Phi$ is not $C^2$ at
the transition.
\end{proof}

\begin{remark}
The integer $k$ is determined by the first nonlinear term in the odd response
$P'$.  The complete two-state model has $k=1$ and the classical exponent
$1/2$, while the constrained three-state model has $k=2$ and exponent $1/4$.
Higher cancellations would produce the hierarchy $1/(2k)$.
\end{remark}

\subsection{Perturbations of the linear pressure}

\begin{proposition}[Stability of response monotonicity]
\label{prop:response-monotonicity-stability}
For an even $C^2$ strictly convex function $P$ on $\R$, set
\begin{equation*}
 \mathcal N_P(y):=P'(y)-yP''(y),\qquad y>0.
\end{equation*}
Fix $k\geq1$ and put
\[
 \omega_k(y):=\frac{y^{2k+1}}{1+y^{2k+1}}.
\]
Suppose that an even $C^2$ strictly convex function $P_0$ satisfies
$P_0'(0)=0$ and, for some $m>0$,
\begin{equation}\label{eq:response-quantitative-margin}
 \mathcal N_{P_0}(y)\geq m\omega_k(y)\qquad(y>0).
\end{equation}
If another even $C^2$ strictly convex function $P$ satisfies $P'(0)=0$ and
\begin{equation}\label{eq:response-weighted-perturbation}
 \sup_{y>0}
 \frac{|\mathcal N_P(y)-\mathcal N_{P_0}(y)|}{\omega_k(y)}<m,
\end{equation}
then $y/P'(y)$ is strictly increasing on $(0,\infty)$. Thus response monotonicity is preserved under sufficiently small weighted perturbations of any reference pressure satisfying \eqref{eq:response-quantitative-margin}.
\end{proposition}

\begin{proof}
Strict convexity and $P'(0)=0$ give $P'(y)>0$ for $y>0$.  Direct
differentiation gives
\begin{equation*}
 \left(\frac{y}{P'(y)}\right)'
 =\frac{\mathcal N_P(y)}{P'(y)^2}.
\end{equation*}
Equations \eqref{eq:response-quantitative-margin} and
\eqref{eq:response-weighted-perturbation} imply
$\mathcal N_P(y)>0$ for every $y>0$, proving the claim.
\end{proof}

\begin{corollary}[Taylor coefficients and critical exponents]
\label{cor:taylor-stratification}
Let $\Theta\subset\R^q$ be open, and let
$\{(T_\theta,\psi_\theta):\theta\in\Theta\}$ be a family such that each
$T_\theta$ is a forward expansive correspondence on a compact metric space
$X_\theta$ and
$\psi_\theta\in C(\Gamma_{T_\theta},\R)$.  Assume that every
$(T_\theta,\psi_\theta)$ satisfies the $C^\omega$ Legendre assumptions,
has unique linear equilibrium pairs, and has a rotation interval symmetric
about $0$. Suppose
that their linear pressures
$P_\theta$ are even, depend
real analytically on $(\theta,y)$ near $\Theta\times\{0\}$, and satisfy the
condition that $y/P_\theta'(y)$ is strictly increasing for $y>0$.  Write
\begin{equation}\label{eq:perturbed-response-series}
 P_\theta'(y)
 =a(\theta)y+\sum_{j=1}^{\infty}c_j(\theta)y^{2j+1},
 \qquad a(\theta)=P_\theta''(0)>0,
\end{equation}
near $y=0$, and, for $k\in\N$, define
\begin{equation*}
 \mathcal S_k
 :=\{\theta\in\Theta:
 c_1(\theta)=\cdots=c_{k-1}(\theta)=0,
 \quad c_k(\theta)<0\}.
\end{equation*}
For $k=1$, there are no vanishing-coefficient conditions.
For the quadratic energy
$\cE_{\beta,\theta}(\gamma)
=\beta(\int\psi_\theta\,d\gamma)^2/2$, the following hold.
\begin{enumerate}[label=\textup{(\roman*)}]
\item On $\mathcal S_k$, the positive equilibrium value has critical exponent
$1/(2k)$, with
\begin{align*}
 \beta_c(\theta)&=\frac1{a(\theta)},
 \\
 z_{\beta,\theta}
 &\sim C(\theta)
   (\beta-\beta_c(\theta))^{1/(2k)},
 \qquad
 C(\theta)
 :=a(\theta)
 \left(\frac{a(\theta)^2}{-c_k(\theta)}\right)^{1/(2k)}.
\end{align*}
Here $z_{\beta,\theta}>0$ is the positive equilibrium value, and the
asymptotic relation holds as $\beta\to\beta_c(\theta)^+$ for fixed $\theta$.
The formulas for $\beta_c$ and $C$ are restrictions of real-analytic
functions on the open set where $a>0$ and $c_k<0$.
In particular, the critical exponent of the positive equilibrium value is
constant on $\mathcal S_k$.
\item If a perturbation of a point in $\mathcal S_k$ creates a first
nonzero coefficient $c_j<0$ with $j<k$, then the critical exponent of the
positive equilibrium value becomes $1/(2j)$. Thus the exponent $1/(2k)$
requires the vanishing of the lower odd Taylor coefficients.
\end{enumerate}
\end{corollary}

\begin{proof}
For $\theta\in\mathcal S_k$, equation
\eqref{eq:perturbed-response-series} has the form
\[
 P_\theta'(y)
 =a(\theta)y-[-c_k(\theta)]y^{2k+1}
  +O(y^{2k+3}).
\]
Part~\textup{(i)} follows from
\cref{thm:taylor-critical-exponent}; analytic dependence follows because
$a>0$ and $-c_k>0$ on $\mathcal S_k$.  More generally, if $c_j$ is the first
nonzero nonlinear coefficient, then
\[
 \mathcal N_{P_\theta}(y)
 =-2j c_j(\theta)y^{2j+1}+O(y^{2j+3}).
\]
Response monotonicity forces $c_j(\theta)<0$.  Applying
\cref{thm:taylor-critical-exponent} with $j$ in place of $k$ proves
part~\textup{(ii)}.
\end{proof}

\begin{remark}
The weight in \eqref{eq:response-weighted-perturbation} behaves as
$y^{2k+1}$ near zero. A perturbation introducing a term $c_jy^{2j+1}$
with $j<k$ therefore makes the quotient in that estimate unbounded near
zero. The monotonicity estimate and the change of exponent in
\cref{cor:taylor-stratification} concern different perturbations.
The exponent is $1/2$ when the first nonlinear coefficient is cubic and
negative; the constrained three-state example has $c_1=0$ and $c_2<0$,
giving exponent $1/4$.
\end{remark}

\begin{remark}
The quantitative margin is satisfied by both continuous-transition examples
below.  For the complete two-state response $P'(y)=\tanh y$,
\[
 \mathcal N_P(y)=\tanh y-y\operatorname{sech}^2y
\]
is positive, has leading term $2y^3/3$, and tends to $1$ at infinity.  For
the constrained three-state response, positivity is exactly the calculation
in \eqref{eq:constrained-monotonicity}; its leading term has order $y^5$
and it also tends to $1$.  Consequently, $\mathcal N_P/\omega_k$ has a
positive infimum with $k=1$ and $k=2$, respectively. The corresponding
weighted perturbations preserve response monotonicity; if the remaining
Legendre and symmetry assumptions persist, the classification in
\cref{thm:taylor-critical-exponent} still applies.
\end{remark}

\begin{remark}
For a fixed primitive finite-state adjacency matrix, a real-analytic family of
two-coordinate observables produces the weighted matrix $B_{\theta,y}$
defined in \eqref{eq:weighted-matrix} below, with $\psi_\theta$ in place of
$\boldsymbol\psi$. Its Perron eigenvalue is simple and real analytic in
$(\theta,y)$.  Hence
$P_\theta(y)=\log\rho(B_{\theta,y})$ has the parameter regularity required in
\cref{cor:taylor-stratification}.  Thus the analytic dependence of pressure required by the corollary is
available for such families; its other assumptions, including symmetry,
Legendre regularity, and response monotonicity, must also hold.  Changing the zero pattern of
the adjacency matrix is a discrete graph perturbation rather than a small
analytic one; such a change may alter the first nonzero nonlinear Taylor coefficient.
\end{remark}

\section{Finite-state correspondences}\label{sec:finite}

\subsection{The Perron--Frobenius formulas}

Let $X=\{1,\ldots,m\}$ and let $A=(a_{ij})$ be a zero-one matrix without zero
rows.  It defines the correspondence
\[
 T_A(i):=\{j:a_{ij}=1\}.
\]
In this finite-state setting, $A$ is the adjacency matrix of the directed
graph with vertex set $X$ and edge set $\Gamma_{T_A}$.  Thus graph-theoretic
terms such as an edge or a forbidden edge refer here to an ordered pair
$(i,j)$ with $a_{ij}=1$ or $a_{ij}=0$, respectively.
A transition matrix $Q$ supported by $A$ has nonnegative entries,
$\sum_{j=1}^mQ(i,j)=1$, and $Q(i,j)=0$ if $a_{ij}=0$.
A probability vector $\mu$ is stationary if
$\sum_{i=1}^m\mu(i)Q(i,j)=\mu(j)$ for every $j$.
The two-coordinate distribution is then
$\gamma_{\mu,Q}(i,j)=\mu(i)Q(i,j)$.
Assume that $A$ is \emph{primitive}, meaning that some positive integer
power of $A$ has every entry strictly positive.  The use of finite-type models
to analyze entropy for Markov set-valued systems is also developed in
\cite{alvin-kelly-markov-set-valued}.  Given a two-coordinate observable
$\boldsymbol\psi:\Gamma_{T_A}\to\R^d$, define
\begin{equation}\label{eq:weighted-matrix}
 B_{\boldsymbol y}(i,j)
 :=\begin{cases}
 \exp(\boldsymbol y\cdot\boldsymbol\psi(i,j)),&a_{ij}=1,\\
 0,&a_{ij}=0.
 \end{cases}
\end{equation}
Here $\rho(B)$ denotes the spectral radius of a matrix $B$.  Primitivity makes
$T_A$ forward expansive and gives the orbit shift the specification property.
Thus every continuous energy considered in this section satisfies the
abundance condition by \cref{cor:specification-abundance}.

For a stationary pair $(\mu,Q)$, the entropy formula is
\[
 h_\mu(Q)=H(\gamma_{\mu,Q})-H(\mu)
 =-\sum_{i,j}\gamma_{\mu,Q}(i,j)\log\gamma_{\mu,Q}(i,j)
   +\sum_i\mu(i)\log\mu(i).
\]
Indeed, the entropy chain rule gives
$H(\mu Q^{[n-1]})=H(\mu)+(n-1)(H(\gamma_{\mu,Q})-H(\mu))$;
dividing by $n$ and taking the limit gives the formula.
Thus $h_\mu(Q)$ is continuous as a function of the two-coordinate
distribution. Such distributions form the compact set of nonnegative
matrices supported by $A$, with total mass one and equal row and column
sums. Maximizing this continuous entropy on the fibres of
$\gamma\mapsto\sum_{(i,j)\in\Gamma_{T_A}}\gamma(i,j)\boldsymbol\psi(i,j)$ shows that
$\mathfrak h$ is upper semicontinuous on the entire rotation set.

\begin{proposition}\label{prop:finite-pf}
For the finite-state correspondence above,
\begin{equation*}
 P(\boldsymbol y)=\log\rho(B_{\boldsymbol y}).
\end{equation*}
Let $r_{\boldsymbol y}$ and $\ell_{\boldsymbol y}$ be positive right and left
Perron--Frobenius eigenvectors, normalized by
$\ell_{\boldsymbol y}\cdot r_{\boldsymbol y}=1$.  The unique linear equilibrium
pair is
\begin{align}
 Q_{\boldsymbol y}(i,j)
  &=\frac{B_{\boldsymbol y}(i,j)r_{\boldsymbol y}(j)}
          {\rho(B_{\boldsymbol y})r_{\boldsymbol y}(i)},\label{eq:finite-kernel}\\
 \mu_{\boldsymbol y}(i)
  &=\ell_{\boldsymbol y}(i)r_{\boldsymbol y}(i).\notag
\end{align}
Moreover,
\begin{equation}\label{eq:finite-gradient}
 \nabla P(\boldsymbol y)
 =\sum_{(i,j)\in\Gamma_{T_A}}\mu_{\boldsymbol y}(i)Q_{\boldsymbol y}(i,j)
    \boldsymbol\psi(i,j).
\end{equation}
If, in addition, $(T_A,\boldsymbol\psi)$ satisfies the $C^1$ Legendre assumptions and $F\in C^1(U)$ on an open neighborhood of
the convex hull of $\boldsymbol\psi(\Gamma_{T_A})$, and if the maximizing set
of $\mathfrak h+F$ is contained in
$\intt\mathcal R(\boldsymbol\psi)$, then the nonlinear
equilibrium pairs for
$\cE(\gamma)=F(\int\boldsymbol\psi\,d\gamma)$ are obtained from the
solutions of
\begin{equation}\label{eq:finite-self-consistency}
 \boldsymbol y
 =\nabla F\bigl(\nabla\log\rho(B_{\boldsymbol y})\bigr).
\end{equation}
Here one retains exactly those solutions for which $\nabla P(\boldsymbol y)$
globally maximizes $\mathfrak h+F$ on $\mathcal R(\boldsymbol\psi)$.
\end{proposition}

\begin{proof}
The proof follows the weighted subshift calculation; compare
\cite[Lemma~3.10]{li-li-zhang-correspondences} and
\cite[Proposition~3.14]{li-li-zhang-correspondences}.  For a
path $(i_0,\ldots,i_n)$, its linear weight is the product
$\prod_{k=0}^{n-1}B_{\boldsymbol y}(i_k,i_{k+1})$.  If $\boldsymbol 1$ is the
all-ones vector, summing over all paths gives
$\boldsymbol 1^{\mathsf T}B_{\boldsymbol y}^n\boldsymbol 1$; the
Perron--Frobenius theorem therefore yields
$P(\boldsymbol y)=\log\rho(B_{\boldsymbol y})$.

Write $\lambda=\rho(B_{\boldsymbol y})$.  The right eigenvector equation
$B_{\boldsymbol y}r_{\boldsymbol y}=\lambda r_{\boldsymbol y}$ shows that
the rows in \eqref{eq:finite-kernel} sum to one.  The left eigenvector equation
gives
\[
 \sum_{i=1}^m\mu_{\boldsymbol y}(i)Q_{\boldsymbol y}(i,j)
 =\frac{r_{\boldsymbol y}(j)}{\lambda}
   \sum_{i=1}^m\ell_{\boldsymbol y}(i)B_{\boldsymbol y}(i,j)
 =\ell_{\boldsymbol y}(j)r_{\boldsymbol y}(j)
 =\mu_{\boldsymbol y}(j),
\]
so $\mu_{\boldsymbol y}$ is stationary.  Positivity and primitivity give the
uniqueness of the corresponding linear equilibrium pair.

For the $q$th coordinate of $\boldsymbol y$, differentiation of the simple
Perron eigenvalue, using
$\ell_{\boldsymbol y}\cdot r_{\boldsymbol y}=1$, gives
\begin{align*}
 \partial_{y_q}P(\boldsymbol y)
 &=\frac{1}{\lambda}
   \ell_{\boldsymbol y}^{\mathsf T}
   (\partial_{y_q}B_{\boldsymbol y})r_{\boldsymbol y}\\
 &=\sum_{(i,j)\in\Gamma_{T_A}}\mu_{\boldsymbol y}(i)Q_{\boldsymbol y}(i,j)
   \psi_q(i,j).
\end{align*}
This proves \eqref{eq:finite-gradient}.  Under the additional hypotheses in
the last part of the statement, \cref{thm:legendre} gives
\eqref{eq:finite-self-consistency}; only solutions whose rotation vectors
maximize $\mathfrak h+F$ give nonlinear equilibrium pairs.
\end{proof}

\subsection{A two-state model}

\begin{proposition}
\label{prop:two-state-transition}
Let $X=\{-1,+1\}$, let $T(s)=X$ for both $s\in X$, and set
$\psi(s,t)=t$.  For $\beta\geq0$, consider the mean-field energy
\begin{equation*}
 \cE_\beta(\gamma)
 :=\frac{\beta}{2}\left(\int_{X^2}\psi\,d\gamma\right)^2.
\end{equation*}
Then the critical inverse temperature is $\beta_c=1$, and the nonlinear
equilibrium pairs are classified as follows.
\begin{enumerate}[label=\textup{(\roman*)}]
\item If $0\leq\beta\leq1$, the unique equilibrium pair is
\[
 \mu_0(s)=\frac12,
 \qquad Q_0(s,t)=\frac12.
\]
\item If $\beta>1$, there are exactly two equilibrium pairs.  Let
$m_\beta\in(0,1)$ be the unique positive solution of
\begin{equation}\label{eq:curie-weiss-equation}
 m_\beta=\tanh(\beta m_\beta).
\end{equation}
For $\sigma\in\{-1,+1\}$, define
\begin{equation*}
 p_{\sigma m_\beta}(t):=\frac{1+\sigma m_\beta t}{2}.
\end{equation*}
The two equilibrium pairs and their two-coordinate distributions are
\begin{equation}\label{eq:two-state-equilibria}
 \mu^\sigma_\beta(s)=p_{\sigma m_\beta}(s),
 \qquad
 Q^\sigma_\beta(s,t)=p_{\sigma m_\beta}(t),
 \qquad
 \gamma^\sigma_\beta(s,t)
 =p_{\sigma m_\beta}(s)p_{\sigma m_\beta}(t).
\end{equation}
\end{enumerate}
The nonlinear pressure is
\begin{equation}\label{eq:two-state-pressure}
 \Ptop^{\cE_\beta}(T)
 =\begin{cases}
   \log 2,&0\leq\beta\leq1,\\
   \log\bigl(2\cosh(\beta m_\beta)\bigr)
      -\dfrac{\beta}{2}m_\beta^2,&\beta>1.
  \end{cases}
\end{equation}
In particular, the model undergoes a continuous symmetry-breaking transition
from a unique equilibrium pair to two distinct equilibrium pairs.
\end{proposition}

\begin{proof}
Order the two states as $-1,+1$.  The weighted matrix in
\eqref{eq:weighted-matrix} is
\begin{align*}
 B_y&=\begin{pmatrix}e^{-y}&e^y\\e^{-y}&e^y\end{pmatrix},\\
 P(y)&=\log\rho(B_y)=\log(2\cosh y),
 &P'(y)&=\tanh y.
\end{align*}
Thus the scalar rotation set is $[-1,1]$.  For $z\in[-1,1]$, its localized
entropy is the binary entropy
\begin{equation}\label{eq:two-state-localized-entropy}
 \mathfrak h(z)
 =-\frac{1+z}{2}\log\frac{1+z}{2}
  -\frac{1-z}{2}\log\frac{1-z}{2},
\end{equation}
where $0\log0:=0$.  This also follows directly from Legendre duality, since
the dual parameter is $y=\operatorname{artanh}z$.

It remains to make the global-maximizer comparison.  Put
\[
 g_\beta(z):=\mathfrak h(z)+\frac\beta2z^2.
\]
On $(0,1)$,
\begin{equation}\label{eq:two-state-derivative}
 g_\beta'(z)=\beta z-\operatorname{artanh}z
 =z\left(\beta-\frac{\operatorname{artanh}z}{z}\right).
\end{equation}
The function $r(z):=\operatorname{artanh}(z)/z$ is strictly increasing from
$1$ to $+\infty$.  Indeed,
\[
 r'(z)=\frac{z/(1-z^2)-\operatorname{artanh}z}{z^2}>0,
\]
because $\operatorname{artanh}z$ is the integral over $[0,z]$ of the strictly
increasing function $(1-t^2)^{-1}$.

If $0\leq\beta\leq1$, equation \eqref{eq:two-state-derivative} is negative
for every $z\in(0,1)$, so the even function $g_\beta$ has the unique global
maximizer $0$.  If $\beta>1$, there is a unique $m_\beta\in(0,1)$ such that
$r(m_\beta)=\beta$, which is equivalent to
\eqref{eq:curie-weiss-equation}.  Moreover, $g_\beta$ increases on
$(0,m_\beta)$ and decreases on $(m_\beta,1)$; by evenness its only global
maximizers are $\pm m_\beta$.  This proves the claimed global classification,
including the endpoints.

For a dual parameter $y$, the Perron--Frobenius formulas give
\[
 \pi_y(t):=\frac{e^{yt}}{2\cosh y},
 \qquad \mu_y(s)=\pi_y(s),
 \qquad Q_y(s,t)=\pi_y(t).
\]
At a maximizing value $z$, the self-consistency relation is
$y=\beta z=\operatorname{artanh}z$, and hence $\pi_y(t)=(1+zt)/2$.
This proves \eqref{eq:two-state-equilibria}.  Finally, the variational value is
$g_\beta(0)=\log2$ below the transition, while at $z=\pm m_\beta$ Legendre
duality gives
\[
 \mathfrak h(z)+\frac\beta2z^2
 =P(\beta z)-\frac\beta2z^2,
\]
which is \eqref{eq:two-state-pressure}.  Since
$\lim_{\beta\to1^+}m_\beta=0$ (monotonically as $\beta$ decreases), the transition is continuous.
\end{proof}

\begin{remark}
The scalar maximization in \cref{prop:two-state-transition} is the classical
Curie--Weiss variational problem considered in
\cite{buzzi-kloeckner-leplaideur-nonlinear}.  Here it is realized for a
finite-state correspondence, and \eqref{eq:two-state-equilibria} additionally identifies
the transition kernels, their stationary measures, and the resulting two-coordinate distributions.  Since $P'(y)=\tanh y=y-y^3/3+O(y^5)$, this is the case
$a=1$, $b=1/3$, and $k=1$ of \cref{thm:taylor-critical-exponent}.
\end{remark}

\subsection{An external field and metastable states}

\begin{proposition}
\label{prop:field-transition}
In the setting of \cref{prop:two-state-transition}, fix $\beta>1$ and, for
$h\in\R$, define the energy with external field $h$ by
\begin{equation*}
 \cE_{\beta,h}(\gamma)
 :=\frac{\beta}{2}\left(\int_{X^2}\psi\,d\gamma\right)^2
   +h\int_{X^2}\psi\,d\gamma.
\end{equation*}
Let $\Phi_\beta(h):=\Ptop^{\cE_{\beta,h}}(T)$.  Then:
\begin{enumerate}[label=\textup{(\roman*)}]
\item At $h=0$, the equilibrium values are $\pm m_\beta$, as in
\cref{prop:two-state-transition}.  For every $h\neq0$, there is a unique
equilibrium value $z_{\beta,h}$, it has the sign of $h$, and
\begin{equation*}
 z_{\beta,h}=\tanh(\beta z_{\beta,h}+h).
\end{equation*}
If $h>0$, then $z_{\beta,h}\in(m_\beta,1)$; if $h<0$, then
$z_{\beta,h}\in(-1,-m_\beta)$.
\item With $p_z(t)=(1+zt)/2$, the unique equilibrium pair for $h\neq0$ and
its two-coordinate distribution are
\begin{equation}\label{eq:external-field-pair}
 \mu_{\beta,h}(s)=p_{z_{\beta,h}}(s),\qquad
 Q_{\beta,h}(s,t)=p_{z_{\beta,h}}(t),\qquad
 \gamma_{\beta,h}(s,t)
 =p_{z_{\beta,h}}(s)p_{z_{\beta,h}}(t).
\end{equation}
\item The function $\Phi_\beta$ is even and convex, and is real analytic on
$(-\infty,0)$ and $(0,+\infty)$.  For $h\neq0$,
\begin{equation}\label{eq:external-field-pressure}
 \Phi_\beta(h)
 =\log\bigl(2\cosh(\beta z_{\beta,h}+h)\bigr)
  -\frac\beta2z_{\beta,h}^2.
\end{equation}
Its one-sided derivatives at the origin are
\begin{equation}\label{eq:first-order-jump}
 \Phi_\beta'(0-)=-m_\beta,
 \qquad
 \Phi_\beta'(0+)=m_\beta.
\end{equation}
Thus $h=0$ is a first-order transition and the equilibrium value jumps by
$2m_\beta$.
\item Set
\begin{equation*}
 s_\beta:=\sqrt{1-\frac1\beta},
 \qquad
 h_{\mathrm{sp}}(\beta)
 :=\beta s_\beta-\operatorname{artanh}s_\beta>0.
\end{equation*}
For $0<|h|<h_{\mathrm{sp}}(\beta)$, the function
$g_{\beta,h}(z)=\mathfrak h(z)+\beta z^2/2+hz$ on $[-1,1]$ has, in addition
to its global maximum, a strict local maximum of the opposite sign and a local
minimum separating the two.  At $|h|=h_{\mathrm{sp}}(\beta)$ the local maximum
and the local minimum coalesce; this is the \emph{spinodal threshold}.
For $|h|>h_{\mathrm{sp}}(\beta)$ the global maximizer is the only critical
point in $(-1,1)$.
\end{enumerate}
\end{proposition}

\begin{proof}
By \eqref{eq:two-state-localized-entropy}, the equilibrium values maximize
\begin{equation*}
 g_{\beta,h}(z)
 :=\mathfrak h(z)+\frac\beta2z^2+hz,
 \qquad -1\leq z\leq1.
\end{equation*}
Suppose first that $h>0$.  For $z>0$,
$g_{\beta,h}(z)>g_{\beta,h}(-z)$, so every global maximizer is positive.  On
$(0,1)$,
\begin{equation}\label{eq:external-field-derivative}
 g_{\beta,h}'(z)=\beta z+h-\operatorname{artanh}z.
\end{equation}
The function $g_{\beta,h}'$ increases up to $s_\beta$ and decreases
thereafter, because
\[
 g_{\beta,h}''(z)=\beta-\frac1{1-z^2}.
\]
Moreover, $g_{\beta,h}'(0)=h>0$, it tends to $-\infty$ as $z\uparrow1$, and
$g_{\beta,h}'(m_\beta)=h$.  Hence it has exactly one positive zero, which lies
in $(m_\beta,1)$, and that zero is the unique global maximizer.  The case
$h<0$ follows by the identity $g_{\beta,h}(z)=g_{\beta,-h}(-z)$.  This proves
part~\textup{(i)}.

At a critical value $z$, the dual parameter is
$y=\beta z+h=\operatorname{artanh}z$.  The Perron--Frobenius calculation in
the proof of \cref{prop:two-state-transition} therefore gives
\eqref{eq:external-field-pair}.  Legendre duality gives
\[
 \mathfrak h(z)+\frac\beta2z^2+hz
 =P(\beta z+h)-\frac\beta2z^2,
\]
which proves \eqref{eq:external-field-pressure}.

The function $\Phi_\beta$ is the supremum of affine functions of $h$, hence is
convex, and it is even because $g_{\beta,h}(z)=g_{\beta,-h}(-z)$.
For $h\ne0$, the unique maximizer satisfies
\[
 g_{\beta,h}''(z_{\beta,h})
 =\beta-\frac1{1-z_{\beta,h}^2}<0.
\]
The analytic implicit function theorem therefore gives analytic dependence
of $z_{\beta,h}$, and hence of $\Phi_\beta$, on each open half-line.
Applying \eqref{eq:maxima-one-sided-derivatives} with $K=[-1,1]$,
$a(z)=\mathfrak h(z)+\beta z^2/2$ and $b(z)=z$ gives
\begin{equation*}
 \Phi_\beta'(h)=z_{\beta,h}\qquad(h\neq0).
\end{equation*}
As $h\to0^+$, the positive solution converges to $m_\beta$; the negative
branch is obtained by symmetry.  This proves \eqref{eq:first-order-jump} and
parts~\textup{(ii)}--\textup{(iii)}.

For the local critical-point structure, rewrite
\eqref{eq:external-field-derivative} as
\begin{equation}\label{eq:spinodal-graph}
 h=H_\beta(z),
 \qquad H_\beta(z):=\operatorname{artanh}z-\beta z.
\end{equation}
The odd function $H_\beta$ has a local maximum at $-s_\beta$ and a local
minimum at $s_\beta$. Its local maximum value is
$H_\beta(-s_\beta)=h_{\mathrm{sp}}(\beta)>0$; positivity follows from
$s/(1-s^2)>\operatorname{artanh}s$ for $s\in(0,1)$.  A horizontal-line
comparison in \eqref{eq:spinodal-graph} now gives three critical points when
$0<|h|<h_{\mathrm{sp}}(\beta)$, a double and a simple critical point at the
threshold, and one critical point beyond it.  Since
$g_{\beta,h}''=-H_\beta'$, the two outer critical points are local maxima and
the middle one is a local minimum.  The sign comparison at the beginning of
the proof identifies the same-sign outer maximum as the unique global one,
leaving the other as metastable.  This proves part~\textup{(iv)}.
\end{proof}

\subsection{The three-state Potts model}

\begin{proposition}
\label{prop:three-state-potts}
Let $X=\{1,2,3\}$, let $T(i)=X$, and define
\[
 \boldsymbol\psi(i,j)
 :=\bigl(\mathbf 1_{\{j=1\}},\mathbf 1_{\{j=2\}}\bigr).
\]
For $\boldsymbol z=(p_1,p_2)$ put $p_3=1-p_1-p_2$, and consider
\begin{equation*}
 \cE_\beta(\gamma)
 :=\frac\beta2\sum_{k=1}^3p_k^2,
 \qquad
 p_k:=\gamma(X\times\{k\}),
 \qquad \beta\geq0.
\end{equation*}
Then the rotation set is the two-dimensional simplex
\begin{equation*}
 \mathcal R(\boldsymbol\psi)
 =\Delta_2:=\{(p_1,p_2):p_1,p_2\geq0,\ p_1+p_2\leq1\},
\end{equation*}
and the critical inverse temperature is
\begin{equation*}
 \beta_c=4\log2.
\end{equation*}
The nonlinear equilibrium pairs are classified as follows.
\begin{enumerate}[label=\textup{(\roman*)}]
\item If $0\leq\beta<\beta_c$, the unique equilibrium probability
vector is
\[
 \boldsymbol p^{\mathrm{dis}}=\left(\frac13,\frac13,\frac13\right).
\]
\item If $\beta=\beta_c$, there are four equilibrium probability
vectors: $\boldsymbol p^{\mathrm{dis}}$ and the three permutations of
\begin{equation}\label{eq:potts-coexistence-vector}
 \left(\frac23,\frac16,\frac16\right).
\end{equation}
\item If $\beta>\beta_c$, there are exactly three equilibrium
probability vectors, namely the permutations of
\begin{equation*}
 (a_\beta,b_\beta,b_\beta),
 \qquad
 a_\beta=\frac{u_\beta}{u_\beta+2},
 \qquad b_\beta=\frac1{u_\beta+2},
\end{equation*}
where $u_\beta>4$ is the unique solution of
\begin{equation*}
 \beta=\frac{(u_\beta+2)\log u_\beta}{u_\beta-1}.
\end{equation*}
\end{enumerate}
For each maximizing vector $\boldsymbol p$, the corresponding equilibrium pair
and two-coordinate distribution are
\begin{equation}\label{eq:potts-equilibrium-pair}
 \mu_{\boldsymbol p}(i)=p_i,
 \qquad Q_{\boldsymbol p}(i,j)=p_j,
 \qquad \gamma_{\boldsymbol p}(i,j)=p_i p_j.
\end{equation}
Moreover, if $\Psi(\beta):=\Ptop^{\cE_\beta}(T)$, then
\begin{equation}\label{eq:potts-pressure}
 \Psi(\beta)=
 \begin{cases}
  \log3+\dfrac\beta6,&0\leq\beta\leq\beta_c,\\[4pt]
  \log(u_\beta+2)-\dfrac{u_\beta}{u_\beta+2}\log u_\beta
  +\dfrac\beta2\dfrac{u_\beta^2+2}{(u_\beta+2)^2},
   &\beta>\beta_c.
 \end{cases}
\end{equation}
The one-sided derivatives satisfy
\begin{equation}\label{eq:potts-pressure-jump}
 \Psi'(\beta_c-)=\frac16,
 \qquad
 \Psi'(\beta_c+)=\frac14,
\end{equation}
so the temperature-driven transition is first order.
\end{proposition}

\begin{proof}
For $\boldsymbol y=(y_1,y_2)$, the weighted matrix has three identical rows,
each equal to $(e^{y_1},e^{y_2},1)$.  Hence
\begin{equation*}
 P(\boldsymbol y)=\log(1+e^{y_1}+e^{y_2}),
\end{equation*}
and its gradient maps $\R^2$ diffeomorphically onto
$\intt\Delta_2$.  The localized entropy is
\begin{equation*}
 \mathfrak h(p_1,p_2)=-\sum_{i=1}^3p_i\log p_i,
 \qquad p_3=1-p_1-p_2.
\end{equation*}
Thus the equilibrium probability vectors maximize
\begin{equation}\label{eq:potts-variational-function}
 J_\beta(\boldsymbol p)
 :=-\sum_{i=1}^3p_i\log p_i+\frac\beta2\sum_{i=1}^3p_i^2
\end{equation}
over the probability simplex.

No boundary point maximizes \eqref{eq:potts-variational-function}: moving a
small mass to a zero coordinate produces an entropy gain of order
$-\eps\log\eps$, which dominates the $O(\eps)$ change of
the quadratic term.  At an interior critical point, Lagrange multipliers give
\begin{equation}\label{eq:potts-lagrange}
 \log p_i-\beta p_i=\log p_j-\beta p_j
 \qquad(1\leq i,j\leq3).
\end{equation}
Since $x\mapsto\log x-\beta x$ is strictly concave, each horizontal line meets
its graph in at most two points; hence a critical vector has at most two
distinct coordinates.  A vector with two equal larger
coordinates cannot be a local maximum.  Indeed, if $a>b$ are its two values,
then \eqref{eq:potts-lagrange} gives
\[
 \beta a=\frac{(a/b)\log(a/b)}{a/b-1}>1,
\]
and the second variation in the direction that separates the two coordinates
equal to $a$ is positive.  Consequently, every nonuniform maximizer is a
permutation of
\begin{equation}\label{eq:potts-u-parametrization}
 \boldsymbol p(u)=\frac1{u+2}(u,1,1),
 \qquad u>1.
\end{equation}

It remains to compare these vectors globally with the disordered vector.  Put
\begin{align*}
 L(u)&:=\log\frac3{u+2}+\frac{u}{u+2}\log u,\\
 C(u)&:=\frac{3(u+2)^2}{(u-1)^2}L(u).
\end{align*}
A direct calculation gives
\begin{equation}\label{eq:potts-global-comparison}
 J_\beta(\boldsymbol p(u))-J_\beta(\boldsymbol p^{\mathrm{dis}})
 =\frac{(u-1)^2}{3(u+2)^2}\bigl(\beta-C(u)\bigr).
\end{equation}
The derivative $C'(u)$ has the sign of
\begin{equation*}
 K(u):=3(u+2)\log\frac{u+2}{3}-(2u+1)\log u.
\end{equation*}
Here $K(1)=K(4)=K'(1)=0$ and
\[
 K''(u)=\frac{(u-1)(u-2)}{u^2(u+2)}.
\]
It follows that $K<0$ on $(1,4)$ and $K>0$ on $(4,+\infty)$.
Consequently, $C$ has its unique minimum at $u=4$, with
\begin{equation}\label{eq:potts-minimum-comparison}
 \min_{u>1}C(u)=C(4)=4\log2.
\end{equation}
Equations \eqref{eq:potts-global-comparison} and
\eqref{eq:potts-minimum-comparison} already prove the assertions below and at
the critical temperature, including the four coexistence vectors.

The vector in \eqref{eq:potts-u-parametrization} is a critical point exactly when
\begin{equation*}
 \beta=B(u):=\frac{(u+2)\log u}{u-1}.
\end{equation*}
The same function $K$ satisfies
\[
 B(u)-C(u)=\frac{u+2}{(u-1)^2}K(u).
\]
Thus a nonuniform critical vector has a larger variational value than the disordered vector only when
$u>4$.  Moreover, $B$ is strictly increasing on $(4,+\infty)$ and maps this
interval onto $(4\log2,+\infty)$.  Indeed, the numerator controlling $B'(u)$
is $R(u)=u+1-2/u-3\log u$; one has $R(4)>0$ and
$R'(u)=(u-1)(u-2)/u^2>0$ for $u>4$.  This proves the uniqueness of $u_\beta$
and the three-phase classification above the transition.

The Perron--Frobenius eigenvectors of the identical-row matrix give
\eqref{eq:potts-equilibrium-pair}.  Substitution of the maximizing vectors in
\eqref{eq:potts-variational-function} gives \eqref{eq:potts-pressure}.
Finally, apply \eqref{eq:maxima-one-sided-derivatives} on the probability
simplex with
\[
 a(\boldsymbol p)=-\sum_{i=1}^3p_i\log p_i,
 \qquad b(\boldsymbol p)=\frac12\sum_{i=1}^3p_i^2.
\]
For $\beta>0$ with $\beta\ne\beta_c$, all maximizing probability vectors differ only by
coordinate permutations and hence have the same value of $b$. Thus
\[
 \Psi'(\beta)=\frac12\sum_{i=1}^3p_i^2
\]
for any maximizing probability vector $\boldsymbol p$.
At $\beta_c$, the disordered vector has $b=1/6$, whereas each ordered vector
in \eqref{eq:potts-coexistence-vector} has $b=1/4$. The left- and
right-derivative formulas give \eqref{eq:potts-pressure-jump}.
\end{proof}

\begin{remark}
This is the three-state mean-field Potts variational problem, expressed through
two state-occupation observables on a correspondence.  In contrast with the scalar example,
the rotation set has nonempty two-dimensional interior and the transition is
first order even with no external field; compare the Potts discussion in
\cite{buzzi-kloeckner-leplaideur-nonlinear}.
\end{remark}

\subsection{Correlations with identical stationary marginals}

\begin{proposition}
\label{prop:edge-correlation}
Let $X=\{-1,+1\}$, let $T(s)=X$, and take the endpoint-dependent
observable
\begin{equation*}
 \psi(s,t):=st.
\end{equation*}
For $\beta\geq0$, define
\begin{equation*}
 \cE_\beta(\gamma)
 :=\frac\beta2\left(\int_{X^2}st\,d\gamma(s,t)\right)^2.
\end{equation*}
Then the nonlinear equilibrium pairs are as follows.
\begin{enumerate}[label=\textup{(\roman*)}]
\item If $0\leq\beta\leq1$, the equilibrium pair is unique and
\begin{equation}\label{eq:edge-correlation-disordered}
 \mu_0(s)=\frac12,
 \qquad Q_0(s,t)=\frac12,
 \qquad \gamma_0(s,t)=\frac14.
\end{equation}
\item If $\beta>1$, let $m_\beta\in(0,1)$ be the positive solution of
$m_\beta=\tanh(\beta m_\beta)$.  There are exactly two equilibrium pairs,
indexed by $\sigma\in\{-1,+1\}$:
\begin{equation}\label{eq:edge-correlation-equilibria}
 \mu_\beta^\sigma(s)=\frac12,
 \qquad
 Q_\beta^\sigma(s,t)=\frac{1+\sigma m_\beta st}{2},
 \qquad
 \gamma_\beta^\sigma(s,t)=\frac{1+\sigma m_\beta st}{4}.
\end{equation}
Their correlations between successive states are
\begin{equation}\label{eq:edge-correlation-order-parameter}
 \int_{X^2}st\,d\gamma_\beta^\sigma(s,t)=\sigma m_\beta.
\end{equation}
\end{enumerate}
The nonlinear pressure is
\begin{equation}\label{eq:edge-correlation-pressure}
 \Ptop^{\cE_\beta}(T)
 =\begin{cases}
   \log2,&0\leq\beta\leq1,\\
   \log\bigl(2\cosh(\beta m_\beta)\bigr)
      -\dfrac\beta2m_\beta^2,&\beta>1.
  \end{cases}
\end{equation}
Both marginals of every $\gamma_\beta^\sigma$ are uniform, but for $\beta>1$
the two-coordinate distributions are nonproduct.  For probability vectors $p,q$ on a finite set $I$, their relative entropy is
\[
 D(p\|q):=\sum_{i\in I}p_i\log\frac{p_i}{q_i},
\]
where zero summands are omitted and $D(p\|q)=+\infty$ if $p_i>0=q_i$
for some $i$. In the present example,
\begin{equation}\label{eq:edge-correlation-information}
 D\bigl(\gamma_\beta^\sigma\,\|\,\mu_\beta^\sigma\otimes
 \mu_\beta^\sigma\bigr)
 =\frac{1+m_\beta}{2}\log(1+m_\beta)
 +\frac{1-m_\beta}{2}\log(1-m_\beta)>0.
\end{equation}
The left-hand side, the relative entropy with respect to the product of
the marginals, is the mutual information of the two successive states.
\end{proposition}

\begin{proof}
Ordering the states as $-1,+1$, the weighted matrix and its linear pressure are
\begin{equation*}
 B_y=\begin{pmatrix}e^y&e^{-y}\\e^{-y}&e^y\end{pmatrix},
 \qquad P(y)=\log(2\cosh y),
 \qquad P'(y)=\tanh y.
\end{equation*}
The Perron--Frobenius formulas give the unique linear equilibrium pair
\begin{equation}\label{eq:edge-correlation-linear-pair}
 \mu_y(s)=\frac12,
 \qquad Q_y(s,t)=\frac{e^{yst}}{2\cosh y}.
\end{equation}
Writing $z=\tanh y$, this becomes
$Q_y(s,t)=(1+zst)/2$.  Its entropy rate, and hence the localized entropy at
correlation $z$, is
\begin{equation*}
 \mathfrak h(z)
 =-\frac{1+z}{2}\log\frac{1+z}{2}
  -\frac{1-z}{2}\log\frac{1-z}{2}.
\end{equation*}
Consequently, the scalar nonlinear variational function is exactly
$\mathfrak h(z)+\beta z^2/2$.  The global calculation in the proof of
\cref{prop:two-state-transition} shows that its unique maximizer is $0$ for
$\beta\leq1$ and that its only maximizers are $\pm m_\beta$ for $\beta>1$.
Substitution in \eqref{eq:edge-correlation-linear-pair} proves
\eqref{eq:edge-correlation-disordered}--\eqref{eq:edge-correlation-order-parameter},
and Legendre duality gives \eqref{eq:edge-correlation-pressure}.

Summing \eqref{eq:edge-correlation-equilibria} over either coordinate gives
the uniform marginal.  The product of these marginals assigns mass $1/4$ to
every edge, so it agrees with $\gamma_\beta^\sigma$ only when $m_\beta=0$.
Direct substitution in the definition of relative entropy gives
\eqref{eq:edge-correlation-information}.
\end{proof}

\begin{remark}
For $z\in[-1,1]$, under the stationary Markov measure with transition probabilities
$Q_z(s,t)=(1+zst)/2$ and uniform stationary measure, let
$(s_n)_{n\geq0}$ denote the coordinate process.  Then
$\mathbb E[s_0s_n]=z^n$, where $\mathbb E$ denotes expectation.  Thus the
correlation is positive at every lag when $z>0$, whereas its sign alternates
with the lag when $z<0$.  Both phases have the same stationary measure on $X$.
The two-coordinate distributions distinguish these phases, although their
one-coordinate marginals agree.
\end{remark}

\subsection{A forbidden transition and exponent \texorpdfstring{$1/4$}{1/4}}

\begin{proposition}
\label{prop:constrained-correlation}
Let $X=\{-1,0,+1\}$, ordered in this way, and let the correspondence be defined
by the primitive non-complete adjacency matrix
\begin{equation*}
 A=\begin{pmatrix}1&1&1\\1&0&1\\1&1&1\end{pmatrix}.
\end{equation*}
Thus the only forbidden edge is $(0,0)$.  Set $\psi(s,t)=st$ and
\begin{equation*}
 \cE_\beta(\gamma)
 :=\frac\beta2\left(\int_{\Gamma_{T_A}}st\,d\gamma(s,t)\right)^2,
 \qquad \beta\geq0.
\end{equation*}
Define
\begin{align}
 \lambda(y)&:=\cosh y+\sqrt{\cosh^2y+2},
 \label{eq:constrained-perron-root}\\
 z(y)&:=\frac{\sinh y}{\sqrt{\sinh^2y+3}}.
 \label{eq:constrained-rotation-parameter}
\end{align}
Then the critical inverse temperature is
\begin{equation*}
 \beta_c:=\sqrt3,
\end{equation*}
and the nonlinear equilibrium pairs are classified as follows.
\begin{enumerate}[label=\textup{(\roman*)}]
\item If $0\leq\beta\leq\beta_c$, there is a unique equilibrium pair,
corresponding to $y=0$ and correlation between successive states $z=0$.
\item If $\beta>\beta_c$, there are exactly two equilibrium pairs,
corresponding to $y=\pm y_\beta$, where $y_\beta>0$ is the unique solution of
\begin{equation}\label{eq:constrained-self-consistency}
 y_\beta
 =\beta\frac{\sinh y_\beta}{\sqrt{\sinh^2y_\beta+3}}.
\end{equation}
Their correlations between successive states are $\pm z_\beta$, where $z_\beta:=z(y_\beta)$.
\end{enumerate}
For an arbitrary dual parameter $y$, the linear equilibrium kernel and its
stationary measure are
\begin{align}
 Q_y&=\begin{pmatrix}
  e^y/\lambda(y)&2/\lambda(y)^2&e^{-y}/\lambda(y)\\
  1/2&0&1/2\\
  e^{-y}/\lambda(y)&2/\lambda(y)^2&e^y/\lambda(y)
 \end{pmatrix},
 \label{eq:constrained-kernel}\\
 \mu_y(-1)&=\mu_y(+1)=\frac{\lambda(y)^2}{2(\lambda(y)^2+2)},
 &\mu_y(0)&=\frac2{\lambda(y)^2+2}.
 \label{eq:constrained-stationary}
\end{align}
In particular, the two nonlinear equilibrium pairs above the transition have
the same stationary measure on $X$, since $\lambda$ is even, but distinct
two-coordinate distributions $\gamma_y(i,j)=\mu_y(i)Q_y(i,j)$.

The nonlinear pressure is
\begin{equation}\label{eq:constrained-nonlinear-pressure}
 \Ptop^{\cE_\beta}(T_A)
 =\begin{cases}
  \log(1+\sqrt3),&0\leq\beta\leq\sqrt3,\\
  \log\lambda(y_\beta)-\dfrac\beta2z_\beta^2,&\beta>\sqrt3.
 \end{cases}
\end{equation}
As $\beta\downarrow\sqrt3$,
\begin{equation}\label{eq:constrained-critical-exponent}
 z_\beta
 \sim\frac{(10\sqrt3)^{1/4}}{\sqrt3}
       (\beta-\sqrt3)^{1/4}.
\end{equation}
Consequently, the transition is continuous: the pressure is
$C^1$ and is not $C^2$ at $\beta=\sqrt3$.
\end{proposition}

\begin{proof}
For the two-coordinate potential $y\psi$, the weighted matrix is
\begin{equation*}
 B_y=\begin{pmatrix}
  e^y&1&e^{-y}\\1&0&1\\e^{-y}&1&e^y
 \end{pmatrix}.
\end{equation*}
Its Perron eigenvalue is \eqref{eq:constrained-perron-root}, with a positive
eigenvector
\begin{equation*}
 r_y=\left(1,\frac2{\lambda(y)},1\right).
\end{equation*}
Since $B_y$ is symmetric, the left eigenvector is proportional to $r_y$.
The formulas in \cref{prop:finite-pf} therefore give
\eqref{eq:constrained-kernel} and \eqref{eq:constrained-stationary}.  They also
give
\begin{equation*}
 P(y)=\log\lambda(y),
 \qquad P'(y)=z(y).
\end{equation*}
In particular, the scalar rotation set is $[-1,1]$.

Let $y(z)$ denote the inverse of the odd increasing map $z(y)$.  Legendre
duality gives $\mathfrak h'(z)=-y(z)$ on $(-1,1)$, so the derivative of the
nonlinear variational function is
\begin{equation}\label{eq:constrained-variational-derivative}
 \frac{d}{dz}\left(\mathfrak h(z)+\frac\beta2z^2\right)
 =\beta z-y(z).
\end{equation}
We claim that $y/z(y)$ is strictly increasing for $y>0$.  Indeed, with
$x=\sinh y$,
\[
 z(y)-yz'(y)
 =\frac{x(x^2+3)-3y\cosh y}{(x^2+3)^{3/2}}.
\]
The numerator vanishes at zero and its derivative is
\begin{equation}\label{eq:constrained-monotonicity}
 3\sinh y\bigl(\sinh y\cosh y-y\bigr)>0
 \qquad(y>0).
\end{equation}
Hence $z(y)/y$ is strictly decreasing, or equivalently $y(z)/z$ is strictly
increasing on $(0,1)$.  Its endpoint limits are
\begin{equation*}
 \lim_{z\to0^+}\frac{y(z)}z=\frac1{P''(0)}=\sqrt3,
 \qquad
 \lim_{z\uparrow1}\frac{y(z)}z=+\infty.
\end{equation*}
The sign pattern in \eqref{eq:constrained-variational-derivative} now proves
the global classification: for $\beta\leq\sqrt3$ the unique maximizer is zero,
whereas for $\beta>\sqrt3$ the only maximizers are $\pm z_\beta$.
At a nonzero maximizer, $y_\beta=\beta z_\beta$, which is
\eqref{eq:constrained-self-consistency}.  Legendre duality then yields
\eqref{eq:constrained-nonlinear-pressure}.

It remains to analyze the onset.  Expansion of
\eqref{eq:constrained-rotation-parameter} at zero gives
\begin{equation*}
 z(y)=\frac{y}{\sqrt3}-\frac{y^5}{30\sqrt3}+O(y^7),
 \qquad
 \frac{y}{z(y)}=\sqrt3\left(1+\frac{y^4}{30}+O(y^6)\right).
\end{equation*}
Solving $\beta=y_\beta/z(y_\beta)$ proves
\eqref{eq:constrained-critical-exponent}. Finally,
\eqref{eq:maxima-one-sided-derivatives}, with $a=\mathfrak h$ on the rotation
interval and $b(z)=z^2/2$, gives
\begin{equation*}
 \frac{d}{d\beta}\Ptop^{\cE_\beta}(T_A)=\frac12z_\beta^2
 \qquad(\beta>\sqrt3),
\end{equation*}
while the derivative is zero below the transition.  By
\eqref{eq:constrained-critical-exponent}, the first derivative is continuous
and the right derivative of this first derivative diverges at $\sqrt3$.  Thus
the pressure is $C^1$ but not $C^2$ there.
\end{proof}

\begin{remark}
For this adjacency matrix, the cubic term in $z(y)$ vanishes, and the first
nonlinear term is of degree five. The resulting exponent $1/4$ differs from
the exponent $1/2$ in the complete two-state model.  In the notation of
\cref{thm:taylor-critical-exponent}, this example has
$a=1/\sqrt3$, $b=1/(30\sqrt3)$, and $k=2$.
\end{remark}

\section{Further directions}

The one-dimensional finiteness theorem leaves a concrete question for several
observables. Fix a primitive finite-state correspondence and
$\boldsymbol\psi:\GammaT\to\R^d$, with $d\geq2$, satisfying the
$C^\omega$ Legendre assumptions, and consider
$F_\beta(\boldsymbol z)=\frac{\beta}{2}\|\boldsymbol z\|^2$.
Is the set of globally maximizing rotation vectors finite for every
$\beta\geq0$, and which parameter values admit coexistence? This is a
correspondence version of the higher-dimensional question raised in
\cite[Section~1.8]{buzzi-kloeckner-leplaideur-nonlinear}. The Potts example
illustrates coexistence for a quadratic energy, but the isolated-zero
argument in \cref{thm:analytic-finiteness} does not give a general answer.

The backward-orbit equidistribution theorem of
\cite[Theorem~C]{li-li-zhang-correspondences} suggests a more specific extension
of \cref{thm:gibbs-ensembles}. Let $T$ satisfy the dynamical hypotheses of
\cite[Theorem~C]{li-li-zhang-correspondences}, and let $\cE$ be a continuous
energy satisfying the abundance condition. Fix an endpoint $x$ and weight each
admissible segment $\boldsymbol x=(x_0,\ldots,x_n)$ with $x_n=x$ by
$\exp(n\cE(L_n(\boldsymbol x)))$. One can ask whether the resulting partition
sum has exponential growth rate $\Ptop^{\cE}(T)$ and whether the weighted
averages of $L_n$ converge. When several equilibrium pairs coexist,
determining which two-coordinate distributions contribute to the limit, and
with what weights, requires more information than the barycenter description
in \cref{thm:gibbs-ensembles}.

Concrete criteria for the Legendre assumptions would also make the reduction
theorem easier to apply. In the single-valued setting,
\cite[Theorem~1.6]{buzzi-kloeckner-leplaideur-nonlinear} gives criteria for
Legendre regularity, while
\cite[Corollary~1.5]{buzzi-kloeckner-leplaideur-nonlinear} gives statistical
properties of each nonlinear equilibrium measure under its stated hypotheses.
For correspondences, a useful next step is to find verifiable conditions on
$T$ and the two-coordinate observables that imply the $C^\omega$ Legendre
assumptions of \cref{def6.2}, and then establish mixing,
decay of correlations, and central limit theorems for the stationary path
measure of each nonlinear equilibrium pair. The holomorphic correspondences
defined by $(w-c)^p=z^q$, with $c\in\mathbb C$ and integers $1\leq p<q$, restricted to their Julia
sets in the parameter regions of
\cite[Theorems~F and G]{li-li-zhang-correspondences}, provide concrete
candidates. These properties concern individual equilibrium pairs; they do
not automatically pass to mixtures arising as Gibbs limits.

Freezing transitions provide another extension beyond the smooth examples
considered here. In \cite[Theorems~4.4 and 4.5]{buzzi-kloeckner-leplaideur-nonlinear},
the equilibrium set becomes independent of the inverse temperature above a
finite threshold and consists of energy-maximizing measures. For
correspondences, one can ask which stationary pairs with $h_\mu(Q)=0$ can be
selected in this way by the family $\beta\cE$, where
$\cE\in C(\Prob(\GammaT),\R)$. A first case
is a correspondence $S$ with $S(x)\subset T(x)$ for every $x\in X$,
whose orbit shift has zero topological entropy. Extending the construction to
arbitrary compact invariant subsets of the orbit space with zero topological
entropy requires care: a continuous potential on that space need not depend
only on two successive coordinates, and its invariant measures need not be
Markov measures. The restriction to two-coordinate distributions, together
with the behavior at boundary maximizing values, therefore needs to be
examined beyond the smooth Legendre analysis above.

Finally, \cref{thm:taylor-critical-exponent,cor:taylor-stratification} describe
critical asymptotics for fixed response coefficients. When a small negative
cubic term perturbs a response whose first nonlinear term is quintic, a joint
limit of the perturbation and the distance to the critical inverse
temperature could describe the passage between the exponents $1/2$ and $1/4$.
Uniform asymptotic formulas should identify the scale on which the two terms
compete, under hypotheses preserving the continuous transition.

\subsection*{Declarations}
\begin{itemize}
\item []\textbf{Ethical approval:} Not applicable.

\item []\textbf{Competing interests:} There is no conflict of interest.

\item []\textbf{Authors' contributions:} D. Tang was responsible for conceptualization and
writing the original draft. Z. Li and R. Yang contributed to supervision and to writing, reviewing,
and editing the manuscript.

\item []\textbf{Funding:} This work is supported by NSFC (No.\ 12671229) and Xi'an International Science and Technology Cooperation Base-Ergodic Theory and Dynamical Systems.

\item []\textbf{Availability of data and materials:} Not applicable.
\end{itemize}
\bibliographystyle{amsplain}
\bibliography{references}

\end{document}